\PassOptionsToPackage{backref}{hyperref}
\documentclass[10pt,centertags]{amsart}

\usepackage{amssymb} 

\usepackage{mathrsfs}             
\usepackage{hyperref}          
\hypersetup{
  colorlinks   = true,
  urlcolor     = blue, 
  linkcolor    = blue, 
  citecolor    = blue
}

\usepackage[usenames,dvipsnames]{xcolor}
\definecolor{refkey}{rgb}{0.8,0.8,0.8}
\definecolor{labelkey}{rgb}{0.9,0,0.1}
\usepackage{calc}               
\usepackage{enumerate}           
\usepackage{tikz-cd}			 
\usepackage{fancyhdr}           
\usepackage{url}                
\usepackage{gloss}              
\usepackage{yhmath}            
\usepackage[capitalize]{cleveref}
\crefname{ineq}{Ineq.}{inequalities}
\creflabelformat{ineq}{#2{\upshape(#1)}#3} 

\usepackage{color}
\usepackage{graphicx}

\usepackage[normalem]{ulem}
\usepackage{todonotes}
\makeatletter
\newtheorem*{rep@theorem}{\rep@title}
\newcommand{\newreptheorem}[2]{%
	\newenvironment{rep#1}[1]{%
		\def\rep@title{#2 \ref{##1}}%
		\begin{rep@theorem}}%
		{\end{rep@theorem}}}
\makeatother

\newtheorem*{main*}{Main Theorem}

\newtheorem{theorem}{Theorem}[section]
\newreptheorem{theorem}{Theorem}
\newtheorem*{theorem*}{Theorem}
\newtheorem{proposition}[theorem]{Proposition}

\newtheorem{lemma}[theorem]{Lemma}
\newtheorem{lem}[theorem]{Lemma}
\newtheorem{corollary}[theorem]{Corollary}
\newreptheorem{corollary}{Corollary}

\newtheorem*{corollary*}{Corollary}
\newtheorem*{cor*}{Corollary}

\newtheorem*{question*}{Question}

\newtheorem*{conjecture*}{Conjecture}
\theoremstyle{definition}

\newtheorem{definition}[theorem]{Definition}

\newtheorem*{definition*}{Definition}
\newtheorem{example}[theorem]{Example}

\theoremstyle{remark}
\newtheorem{remark}[theorem]{Remark}

\numberwithin{equation}{section}
\newcommand{\C}{\mathbb{C}}
\newcommand{\R}{\mathbb{R}}
\newcommand{\Z}{\mathbb{Z}}
\newcommand{\V}{\mathbb{V}}

\newcommand{\N}{\mathbb{N}}

\newcommand{\mc}{\mathcal}

\newcommand{\bb}{\mathbb}
\newcommand{\mf}{\mathfrak}

\newcommand{\Ga}{\Gamma}
\newcommand{\ga}{\gamma}
\newcommand{\eps}{\epsilon }

\DeclareMathOperator{\cd}{cd}

\DeclareMathOperator{\tr}{tr}

\DeclareMathOperator{\supp}{supp}

\DeclareMathOperator{\vol}{Vol}

\DeclareMathOperator{\Jac}{Jac}

\DeclareMathOperator*{\bary}{bar}

\DeclareMathOperator{\Sp}{\Sp}

\newcommand{\op}[1]{\operatorname{#1}}
\newcommand{\set}[1]{\left\{#1 \right\}}
\newcommand{\til}{\widetilde}

\newcommand{\floor}[1]{\left\lfloor #1 \right\rfloor}

\newcommand\rest[1]{\raisebox{-.5ex}{$|$}_{#1}}

\newcommand{\pa}{\partial }

\newcommand{\of}{\circ }
\providecommand{\to}{\longrightarrow }

\newcommand{\abs}[1]{\left\lvert #1 \right\rvert }
\newcommand{\norm}[1]{\left\| #1 \right\| }
\newcommand{\inner}[1]{\left\langle #1 \right\rangle }

\DeclareMathOperator{\CAT}{\mathsf{CAT}}

		\renewcommand{\bar}{\overline}

\newcommand{\cout}[1]{}
\definecolor{darkcyan}{rgb}{0., 0.65, 0.65}

\begin{document}


\title{The natural flow and the critical exponent}
\author{Chris Connell, D. B. McReynolds, Shi Wang}

\begin{abstract}
Inspired by work of Besson--Courtois--Gallot, we construct a flow called the \emph{natural flow} on a non-positively curved Riemannian manifold $M$. As with the natural map, the $k$--Jacobian of the natural flow is directly related to the critical exponent $\delta$ of the fundamental group. There are several applications of the natural flow that connect dynamical, geometrical, and topological invariants of the manifold.  First, we give $k$--dimensional linear isoperimetric inequalities when $k > \delta$. This, in turn, produces lower bounds on the Cheeger constant. We resolve a recent conjecture of Dey--Kapovich on the non-existence of $k$--dimensional compact, complex subvarieties of complex hyperbolic manifolds with $2k > \delta$. We also provide upper bounds on the homological dimension, generalizing work of Kapovich and work of Farb with the first two authors. Using the natural flow together with Morse theory, we also give upper bounds on the cohomological dimension, which partially resolve a conjecture of Kapovich. Finally, we introduce a new growth condition on the Bowen--Margulis measure that we call \emph{uniformly exponentially bounded} that we connect to the cohomological dimension and which could be of independent interest. 
\end{abstract}


\maketitle


\section{Introduction}

In this paper, we will explore connections between different invariants of groups and manifolds arising from a Hadamard space. Recall that a \textbf{Hadamard space} is a connected, simply connected, complete Riemannian manifold $(X,d)$ with non-positive sectional curvatures. Given a finitely generated, discrete, torsion-free subgroup $\Gamma < \mathrm{Isom}(X)$, we have the associated Riemannian manifold $M= X/\Gamma$. Given a point $p \in X$, the \textbf{critical exponent of $\Gamma$} (\textbf{or $M$}) is defined to be
\[\delta(M) = \delta(\Ga):=\inf\{s\,:\,\sum_{\ga\in \Ga}e^{-sd({p},\ga {p})}<\infty\}.\]

The main theme of this paper is that the critical exponent can be used to control a wide array of different geometrical and topological properties of $M$. For instance, we show that if $k > \delta$, then $M$ satisfies a $k$--dimensional linear isoperimetric inequality. In the special case of codimension one, this provides bounds on the Cheeger constant. We provide upper bounds on the homological dimension of $M$ via the critical exponent which generalizes work of Kapovich and work of Farb with the first two authors of this paper. Using different methods, we provide upper bounds for the cohomological dimension which partially resolves a conjecture of Kapovich. We also resolve a recent conjecture of Dey--Kapovich on the non-existence of compact complex subvarieties of complex hyperbolic manifolds whose dimension is controlled by the critical exponent. Each of these applications requires different techniques and methods in order to establish it. We employ a number of tools from areas including geometric measure theory, algebraic topology, Lie theory, dynamics, and complex geometry. The main tool used to prove these results is the natural flow.  

\subsection{The Natural Flow} 

Our flow is inspired by the pioneering work of Besson--Courtois--Gallot \cite{BCG95, BCG99}. Although Besson--Courtois--Gallot treat a much more general setting of maps $f\colon M \to N$, we consider the simplest case where we have the identity map $\mathrm{Id}\colon M \to M$ and $M$ is non-positively curved. The \textbf{natural map} or \textbf{barycenter map} $F\colon M \to M$, which is homotopic to $\mathrm{Id}$, is defined by $F(x) = \mathrm{bar}(\mu_x)$ where $\mu_x$ is the Patterson--Sullivan measure and $\mathrm{bar}(\star)$ is a certain choice of barycenter of the measure $\star$. The approach we take here is a continuous version of this in the form of a flow that we call the \textbf{natural flow} or \textbf{barycenter flow}. The idea is quite simple. Given any $\delta$--conformal density $\{\mu_x:x\in X\}$ on $X$ associated to $\Ga$, we take the gradient flow of the potential function $\bar{f}\colon X \to \R$ given by
\begin{equation}\label{eq:potential}
\bar f(x)=-\log ||\mu_x||.
\end{equation}
The function $\bar{f}$ is $\Ga$--invariant and therefore descends to a function $f\colon M\to \R$. The gradient field $\nabla f=\mathfrak X$ integrates to a $C^1$--flow $\phi_t$ on $M$ which will be complete. The natural flow gives a family of diffeomorphisms which reflect properties of the natural map, which in general is not a diffeomorphism. One specific connection is that the fixed points of the natural map and the natural flow coincide. Through the results of this paper, we hope to convince the reader of the broad utility of this flow. We expect that there will be additional applications beyond those given here in areas like higher Teichm\"{u}ller theory, the study of thin groups, or more generally the study of Zariski dense subgroups of semisimple Lie groups. We will provide further applications along with some problems motivated by this present work in \cite{CMW26}. 

An important feature of the natural map is the existence of upper bounds on the $k$--Jacobian that depend on $X$ and the critical exponent of $\Gamma$. The bounds for the natural map were established first by Besson--Courtois--Gallot \cite{BCG99} in the real rank one setting and sharpened by Farb with the first two authors of this paper \cite{CFM19}. The estimates were rather technical to establish.

The starting point is a continuous version of the $k$--Jacobian bounds. One advantage of the  continuous version is that the $k$--Jacobian estimate is converted into a $k$--trace estimate. Consequently, these bounds are better, more general, and substantially easier to establish. We view this as one of the \emph{especially nice} features of the natural flow. In fact, all this comes down to the following general observation which is presumably known to experts. In the statement, the $k$--trace $\mathrm{tr}_k(\star)$ of a bilinear form $\star$ is the sum of the $k$ smallest eigenvalues of $\star$ (see Definition \ref{def:k-trace}), with respect to the Riemannian inner product.

\begin{theorem}[Uniform Exponential Contraction] \label{prop:phi_contracts}
Let $M$ be a Riemannian manifold, let $f\colon M \to \R$ be a $C^2$ function whose gradient vector field $\nabla f = \mathfrak{X}$ is complete, and let $\phi_t\colon M \to M$ be the associated one-parameter group of self-diffeomorphisms. If $k >0$ and $\epsilon_k \in \R$ satisfy $\op{tr}_k(\nabla df)(x)\geq\epsilon_k$ for all $x\in M$, then
\[\op{Jac}_k \phi_{-t}(x)\leq e^{-\epsilon_k t},\]
for all $t>0$, $x\in M$. In particular, if $\epsilon_k>0$ then $\phi_{t}$ uniformly exponentially expands the $k$--volume and $\phi_{-t}$ uniformly exponentially contracts the $k$--volume as $t$ increases.	
\end{theorem}

Applying Theorem \ref{prop:phi_contracts} to the setting of pinched negative curvature and where $f$ is given by (\ref{eq:potential}), we can relate $\epsilon_k$ to the critical exponent. 

\begin{corollary}[Uniform Exponential Contraction: Negative Curvature]\label{T:MainJac} 
If $X$ is an $n$--dimensional, simply connected, pinched negatively curved manifold whose sectional curvatures are bounded above by $-1$, and $\Ga < \mathrm{Isom}(X)$ is a discrete, torsion-free subgroup with critical exponent $\delta$, then there exists a one-parameter family of self-diffeomorphisms $F_t$ ($=\phi_{-t}$) on the quotient manifold $M=X/\Ga$, which satisfies
\[\op{Jac}_k F_t(x)\leq e^{-\delta(k-1-\delta)t},\]
for any $x\in M$, $t>0$ and integer $k\in [1,n]$.
\end{corollary} 

\subsection{Isoperimetric inequalities and complex subvarieties}

The classical Plateau's problem asks: what is the area minimizing surface enclosed by a closed curve $c$ in $\mathbb R^3$? Despite existence/regularity issues of the area minimizer, the minimizing area is always well-defined. Bounding the area of the minimizer as a function of the length of $c$ is the general isoperimetric problem and makes sense in higher dimensions and in arbitrary manifolds (possibly with non-trivial topology). Using the natural map in pinched negatively curved manifolds, Liu and the third author \cite{LiuWang20} proved that for any smooth $1$--boundary, the minimizing area of the surfaces is always linearly bounded by the length of the $1$--boundary, provided the critical exponent is small. Using the natural flow, we can sharpen this bound and extend it to higher dimensions. For simplicity, we only state the theorem for compact immersed submanifolds though the argument works for finite volume smoothly parameterized manifolds; a smooth $k$--boundary is boundary of a compact smoothly immersed $(k+1)$--manifold. 

\begin{theorem}[Isoperimetric Inequality]\label{T:MainIso}
If $X$ is an $n$--dimensional, simply connected, pinched negatively curved manifold whose sectional curvatures are bounded above by $-1$, and $\Ga < \mathrm{Isom}(X)$ is a discrete, torsion-free subgroup with critical exponent $\delta$ such that $\delta<k$ for some positive integer $k\in [1,n-1]$, then for any smooth $k$--boundary $B$ in the quotient manifold $M=X/\Ga$, and any $\epsilon>0$, there exists a smoothly immersed $(k+1)$--manifold $\Omega^\epsilon$, such that $\partial \Omega^\epsilon=B$ and
\[\vol_{k+1}(\Omega^\epsilon)\leq \frac{1}{k-\delta}\vol_k(B)+\epsilon.\]
\end{theorem}

Apply Theorem \ref{T:MainIso} to the special case $k=n-1$, we obtain a lower bound on the Cheeger constant.

\begin{corollary}[Buser Type Inequality]\label{Cor:Cheeger}
If $X$ is an $n$--dimensional, simply connected, pinched negatively curved manifold whose sectional curvatures are bounded above by $-1$, and $\Ga < \mathrm{Isom}(X)$ is a discrete, torsion-free subgroup with critical exponent $\delta$, then the Cheeger constant $h(M)$ of the quotient manifold $M=X/\Ga$ satisfies $h(M)\geq n-1-\delta$.
\end{corollary}

When $M$ is a real hyperbolic manifold, Sullivan \cite{Sullivan87} relates $\lambda_0$ and $\delta$ by 
\[\lambda_0=\begin{cases}
\frac{(n-1)^2}{4} \quad\text{ if } \delta\leq\frac{n-1}2,\\
\delta(n-1-\delta)\quad\text{ if } \frac{n-1}2<\delta\leq n-1.\\
\end{cases}
\]
By Cheeger's inequality, we have $h(M)\leq 2\sqrt{\lambda_0}\leq n-1$, and thus Corollary \ref{Cor:Cheeger} is sharp when $\delta\rightarrow 0^+$. On the other hand, for a classical Schottky group $\Ga<\op{Isom}(\mathbf{H}_\R^n)$ ($n\geq 3$), there exists $\eps_n>0$ such that $\delta(\Ga)\leq n-1-\eps_n$ by Phillips--Sarnak \cite[Prop 3.10]{PhillipsSarnak} ($n\geq 4$) and Doyle \cite{Doyle} ($n=3$). Combining this with our work, we obtain an explicit uniform lower bound $\eps_n$ on the Cheeger constant in this setting. This result was further extended by Bowen \cite{Bowen15} for $n\geq 4$ and even, or Kleinian groups which are subgroups of $3$--dimensional hyperbolic lattices.

We also can resolve a conjecture of Dey--Kapovich \cite[Conj 20]{Dey-Kapovich} (see also \cite[Conj 10.3]{Kapovich22}) on the non-existence of compact, complex subvarieties of complex hyperbolic manifolds. 

\begin{theorem}[Dey--Kapovich Conjecture]\label{Thm:DeyKapConj}
If $M$ is a complete, complex hyperbolic $n$--manifold with $\delta(M)<2k$ for some positive integer $k\leq n$, then there exists no compact complex $k$--dimensional subvarieties of $M$.
\end{theorem}

The proof of this uses a refined version of Corollary \ref{T:MainJac} when restricting to holomorphic tangent subspaces (Lemma \ref{lem:holomorphic_contract}) and that compact, complex subvarieties are volume minimizing.  

\subsection{Homological Dimension}

In this section and the next, we describe applications on the topology of $X/\Gamma$ via homology and cohomology. The methods we use are different for each but both make essential use of the natural flow and Patterson--Sullivan theory. Before stating explicitly the results, we briefly describe the methods used. As to why we need to use different methods for each, note that as $X/\Gamma$ will generally be non-compact (in fact infinite volume), the homology is dual to the compactly supported cohomology while the cohomology is dual to the local finite homology. The latter involves the topology at infinity and so is more delicate to control.

To establish the necessary homological vanishing results, using Theorem \ref{prop:phi_contracts}, we see that the natural flow is contracting in sufficiently high dimensions; this dimensional threshold is given by the critical index. In this case, if we have a $k$--dimensional cycle, by pushing it forward under the flow, we see that the mass/volume of the cycle under the flow goes to zero exponentially fast. It follows then from Gromov's mass vanishing theorem (see Theorem \ref{thm:small_vanishing} below) that the cycle must be trivial. This should be viewed as a continuous version of the iterative method introduced by Kapovich \cite{Kapovich}. His method uses iterations of the natural map and estimates on the $k$--Jacobian to prove that the mass of the cycle goes to zero exponentially fast under the push forward of the $n$--iteration. As we use  Theorem \ref{prop:phi_contracts}, we require $k$--trace bounds and so the critical index is defined via $k$--traces. In the case when $X$ is a symmetric space, we estimate this critical index via Patterson--Sullivan theory with conformal densities playing an important role. 

For the cohomology, we use a proper Morse perturbation of our potential function and Morse theory in combination with Patterson--Sullivan theory to establish the necessary vanishing. We do that by bounding the index of the critical points that arise; this is again given by the critical index. We then get bounds on the dimension of the cells in our cell decomposition and hence obtain control on the topology of $X/\Gamma$. In order to ensure we can perturb our potential to a proper Morse function with both a lower bound and indices controlled at the critical points, we need a growth condition on the norms of the Patterson--Sullivan measures. We show that when the associated Bowen--Margulis measure on the unit tangent bundle is finite, then our growth condition is satisfied (see Theorem \ref{thm:finiteness criterion}), and without these growth conditions, the vanishing results cannot always hold. We consider the question of which groups satisfy the growth condition needed for the perturbation to be of independent interest to those working in Patterson--Sullivan theory. Unlike in the case of homology, our cohomological vanishing results are the first of this kind that do not make use of homological dimension.

\subsubsection{Hadamard Spaces}

Given a commutative ring with identity $R$ and a torsion-free group $\Gamma$, the \textbf{homological dimension of $\Ga$} is defined to be
\[\op{hd}_R (\Ga):=\inf\{s\in \mathbb Z_+\,:\, H_i(\Ga;V)=0 \text{ for any } i>s \text{ and any $R\Ga$--module $V$}\},\]
and similarly the \textbf{cohomological dimension of $\Ga$} is defined as
\[\op{cd}_R (\Ga):=\inf\{s\in \mathbb Z_+\,:\, H^i(\Ga;V)=0 \text{ for any } i>s \text{ and any $R\Ga$--module $V$}\}.\]
As a convention, when $R=\mathbb Z$ we simply omit the subscripts. These dimensions characterize the algebraic/topological size $\Gamma$ or $M$ when $\Gamma = \pi_1(M)$.  We remark that $H_k(\Gamma;V)=H_k(\Omega;\V)$ for any CW complex $\Omega$ which is a $K(\Gamma,1)$ and a certain flat bundle $\V$ over $\Omega$ (see \S \ref{sec:hom_prelim}).

Following Kapovich \cite{Kapovich}, we explore the relationship between $\mathrm{cd}_R(\Gamma)$ or $\mathrm{hd}_R(\Gamma)$ and $\delta(\Gamma)$. When $X=\mathbf{H}_\R^n$, \cite{Kapovich} established $\op{hd}_R(\Ga)\leq \delta(\Ga)+1$. Farb with the first two authors \cite{CFM19} extended this work to the real rank one setting (see also \cite[Cor 5.6]{Carron-Pedon04}). Iterations of the natural map were used in both \cite{Kapovich} and \cite{CFM19} to establish these results in combination with estimates on the $k$--Jacobian. One of the main purposes of this paper is to extend these results to arbitrary manifolds with bounded sectional curvatures. With Theorem \ref{prop:phi_contracts}, we obtain the following relative homological vanishing result. It is worth noting that we do not require $\Gamma$ be finitely generated in our vanishing result. 

\begin{theorem}[Relative Homological Vanishing]\label{thm:main}
If $X$ is a Hadamard space with bounded sectional curvatures, $\Gamma < \mathrm{Isom}(X)$ is a discrete, torsion-free subgroup with $M = X/\Gamma$, and $\mathbb{V}$ is a flat bundle over $M$, then for every $\epsilon > 0$ and $k\geq  j_X(\Gamma)$, the homomorphism $H_k(M_{\eps};\mathbb V)\rightarrow H_k(M;\mathbb V)$ induced by inclusion is surjective.
\end{theorem}

Here $M_{\eps}$ is the $\eps$--thin part of $M$ defined by $M_{\eps}=\set{x\in M\,:\, \op{injrad}(x)<\eps}$ where $\op{injrad}(x)$ is the injectivity radius of $M$ at $x$. The value $j_X(\Gamma)$ is called the \textbf{critical index of $\Gamma$}. It is computable in terms of the geometry of $X$ and the asymptotic behavior of the $\Gamma$--action (see Definition \ref{D:CritIndex}). One can view this result as saying the homology of $\Ga$ vanishes when the degree is at least $j_X(\Ga)$ relative to the homology of the ``zero-thin part at infinity.'' We also obtain an explicit relative vanishing result that extends the relative vanishing result of Kapovich \cite{Kapovich} (see Theorem \ref{thm:relative-vanishing}).

As colimits preserve epimorphisms, we have the following corollary.

\begin{corollary}\label{cor:hd_main}
If $X$, $\Gamma$, $\mathbb V$ and $M$ are as in the statement of Theorem \ref{thm:main}, then 
\[ \lim_{\eps\rightarrow 0} H_k(M_{\eps};\mathbb V)\rightarrow H_k(M;\mathbb V) \] 
is surjective whenever $k\geq  j_X(\Ga)$. Here $\lim_{\eps\rightarrow 0}$ denotes the colimit (direct limit) with respect to the system $\{M_{\eps}:\eps>0\}$ under inclusions. 
\end{corollary}

We also have the following corollary (see Corollary \ref{cor:hd_bound} and Corollary \ref{cor:neg_curved}).

\begin{corollary}[Homological Vanishing]\label{cor:hd_main2}
Let $X$, $\Gamma$ be as in Theorem \ref{thm:main}. If either 
\begin{enumerate}
\item $X$ has pinched negative curvature and $\Ga$ has no parabolic fixed point, or
\item the injectivity radius of $X/\Ga$ is uniformly bounded from below,
\end{enumerate}
then $\op{hd}_R(\Gamma)\leq j_X(\Ga)-1$.
\end{corollary}

\subsubsection{Symmetric Spaces}

When $X=\mathbf{H}_\R^n$ (resp. $K_X\leq -1$), then $j_X(\Ga)=\floor{\delta(\Ga)}+2$ (resp. $j_X(\Ga)\leq\floor{\delta(\Ga)}+2$). This recovers \cite{Kapovich} and most of \cite{CFM19}. More precisely, we prove:

\begin{corollary}[Homological Vanishing: Real Rank One]\label{cor:rank one} 
Let $X=G/K$ be a rank one symmetric space of dimension $n$ such that all sectional curvatures are constant $-1$ or all are in $[-4,-1]$, and let $\Gamma<G$ be any discrete, torsion-free subgroup with no parabolics.
\begin{enumerate}
\item 
If $X = \mathbf{H}_\R^n$ is real hyperbolic $n$--space, then $\op{hd}_R(\Gamma)\leq \delta+1$.
\item 
If $X = \mathbf{H}_\C^n$ is complex hyperbolic $n$--space ($n \geq 2$), then
\[ \op{hd}_R(\Gamma)\leq \begin{cases} \delta+1 & \delta< 2n-2 \\ 2n-1 & 2n-2\leq \delta< 2n \\ 2n & \delta=2n\\  \end{cases}. \]
\item 
If $X = \mathbf{H}_{\mathbb{H}}^n$ is quaternionic hyperbolic $n$--space ($n \geq 2$), then
\[ \op{hd}_R(\Gamma)\leq \begin{cases} \delta+1 & \delta< 4n-4 \\ 4n-3 & 4n-4\leq \delta< 4n-2\\ 4n-2 & 4n-2\leq \delta< 4n \\ 4n & \delta=4n+2\\ \end{cases}. \]
\item 
If $X = \mathbf{H}_{\mathbb{O}}^2$ is the Cayley hyperbolic plane, then
\[ \op{hd}_R(\Gamma)\leq \begin{cases}  \delta+1 & \delta< 8 \\ 9 & 8\leq \delta<10\\ 10 & 10\leq \delta< 12 \\ 11 & 12\leq \delta< 14\\ 12 & 14\leq \delta<16\\ 16 & \delta=22\\ \end{cases}. \]
\end{enumerate}
\end{corollary}

For higher-rank symmetric spaces, the estimate on the critical index is more involved and will be discussed in \cite{CMW26}.

\subsection{Cohomological Dimension}

To control the cohomology, we use Morse theory to analyze the topology of the quotient manifold. The potential function \eqref{eq:potential} for the natural flow is $C^2$ and can therefore be perturbed to be Morse. For our specifically chosen Morse function, we can control the possible indices in terms of the critical exponent. However, since our manifold is non-compact, we also need our function to be proper and bounded from below. To ensure properness, a second perturbation is required (see Corollary \ref{cor:perturb_Morse}) that also maintains index control. To obtain a uniform lower bound, we will also need some conditions on $\Gamma$ (see Remark \ref{Rem:LowerBoundMorse}) in terms of the norm growth of $\mu_x$. Given a conformal density $\mu$, we say $||\mu_x||$ has \textbf{subexponential growth}, if for any $\eta>0$, there exists $C(\eta)>0$ such that
\[||\mu_x||\leq C(\eta)e^{\eta \cdot d(x,p)},\;\text{ for all } x\in M,\]
where $p\in M$ is some chosen basepoint. We prove that,

\begin{theorem}[Cohomological Vanishing: Subexponential Growth]\label{thm:cd}
If $X$ is a Hadamard space and $\Ga<\op{Isom}(X)$ is a discrete, torsion-free subgroup with a conformal density $\mu$ such that $\norm{\mu_x}$ has subexponential growth, then $\cd(\Ga)\leq j_X(\mu)-1$.
\end{theorem}

Here $j_X(\mu)$ is the \textbf{critical index associated to the conformal density $\mu$} (see Definition \ref{D:CritIndex}). When $X$ is pinched negatively curved with sectional curvatures normalized to be in $(-\infty,-1]$ and $\mu$ has dimension $\delta(\Ga)$, we have that $j_X(\mu)\leq\floor{\delta(\Ga)}+2$. 

Whether $\norm{\mu_x}$ has subexponential growth is more mysterious. If $\Ga$ is convex-cocompact, then there exists a unique $\delta$--conformal density $\mu$ for $\delta=\delta(\Ga)$, and for this $\mu$, one can show that $||\mu_x||$ is uniformly bounded. On the other hand, as is noted in Corollary \ref{cor:rank one}, the inequality fails in the presence of parabolic subgroups of sufficiently large rank compared to $\delta$, and by our theorem $||\mu_x||$ must increase exponentially along a ray into a cusp. Sullivan first observed this phenomenon in geometrically finite manifolds of constant negative curvature \cite{Sullivan84}. Kapovich conjectured \cite[Conj 1.4]{Kapovich} that when $X=\mathbf{H}_\R^n$, the cohomological dimension relative to the parabolic subgroups of rank at least 2 is bounded above by $\delta(\Ga) + 1$. Therefore, our theorem partly answers this in the affirmative for a large class of Kleinian groups including the convex cocompact subgroups for which this was previously known.

When $X$ has pinched negative curvature, the $\delta$--conformal density $\mu_x$ (for $\delta=\delta(\Ga)$) naturally corresponds to a geodesic flow invariant measure $\nu$ on the unit tangent bundle $T^1M$, which is known as the Bowen--Margulis measure. When $\nu$ is finite, the manifold enjoys good dynamical properties. In particular, the $\delta(\Ga)$--conformal density $\mu$ is unique, $\Ga$ is of divergence type and the geodesic flow is ergodic and conservative \cite[\S 5]{Yue96}. For convex cocompact manifolds, $\nu$ is always finite. Peign\'e \cite{Peigne} gave some geometrically infinite examples where $\nu$ is finite as well. We connect the finiteness of $\nu$ to the boundedness of $\mu_x$, in which case our Theorem \ref{thm:cd} can be applied.

\begin{theorem}[Uniformly Bounded Norms]\label{thm:finiteness criterion}
Suppose $X$ has pinched negative curvature and $M=X/\Gamma$ is a manifold whose injectivity radius is bounded away from zero. If $\nu(T^1M)<\infty$, then the corresponding unique $\delta(\Ga)$--conformal density $\mu_x$ has uniformly bounded norm on $M$.
\end{theorem}

Finally, we obtain the following corollary. 

\begin{corollary}[Cohomological Vanishing]\label{cor:CDBound}
Suppose $X$ has pinched negative curvature and $M=X/\Gamma$ is a manifold whose injectivity radius is bounded away from zero. If $\nu(T^1M)<\infty$, then $\op{cd}(\Ga)\leq j_X(\mu)-1$, 
where $\mu$ is the unique $\delta(\Ga)$--conformal density associated to $\nu$. Furthermore, if $X$ is normalized to be $K\leq -1$, then $\op{cd}(\Ga)\leq \floor{\delta(\Ga)}+1$.
\end{corollary}

\subsection*{Outline of Paper.}

We give a brief outline of the paper. In \S \ref{sec:gradient_flow} we present the natural flow and the proof of its contraction properties. We also present the proof of Theorems \ref{Thm:DeyKapConj} and \ref{T:MainIso} and Corollary \ref{Cor:Cheeger} as applications. In \S \ref{sec:homdim}, we provide background for the homological results and prove Theorem \ref{thm:main}. We then apply this theorem to the special case of rank one locally symmetric spaces. In \S \ref{sec:cohomdim} we establish \cref{thm:finiteness criterion} and the cohomological bounds in \cref{thm:cd} via Morse theory. Finally, in a short appendix, we give the proof of Theorems \ref{thm:small_vanishing}.

\subsection*{Acknowledgements.}
This work was partially supported by grants Simons Foundation/SFARI-965245 (CC) and NSF DMS-1812153 (DBM) and NSFC-12301085 (SW). The authors thank the Max Planck Institute of Mathematics-Bonn (CC and SW) and the Karlsruhe Institute of Technology (CC), where some of this work was completed. We would also like to thank David Fisher, Gabrielle Link, Beibei Liu, Ben Lowe, Hee Oh, Jean-Fran\c{c}ois Quint, Roman Sauer, and Bingyu Zhang for helpful conversations on the material in this paper.

\section{The natural flow}\label{sec:gradient_flow}

We will assume $X$ is a Hadamard space of sectional curvature $-a^2\leq K\leq 0$ with metric $d$ and $\Ga<\op{Isom}(X)$ is a discrete, torsion-free subgroup. Let $\partial_\infty X$ denote the geodesic boundary of $X$. The \textbf{Busemann} function on $X$ is defined by
\[B_\theta(x,y)=\lim_{t\rightarrow \infty}\left(d(x,c_{y\theta}(t))-t\right),\quad \forall x,y\in X,\;\forall \theta\in \partial_\infty X,\]
where $c_{y\theta}\colon \R^+\rightarrow X$ is the geodesic ray connecting $y$ to $\theta$. When a base point ${p}\in X$ is fixed, we simply denote $B_\theta(x)=B_\theta(x,{p})$. For a Hadamard space $X$, it is known \cite{Heintze-Im-Hof} that for each fixed $\theta\in \partial_\infty X$, the Busemann function $B_\theta(x)$ is convex and $C^2$ on $X$.

\begin{definition}\label{def:delta_conf_density}
Given $\delta> 0$, we say a family of finite nonzero Borel measures $\{\mu_x:x\in X\}$ is a \textbf{$\delta$--conformal density on $X$} (associated to the $\Gamma$--action) if
\begin{enumerate}
\item each $\mu_x$ is supported on $\supp{\mu_x}\subset\partial_\infty X$,
\item $\mu_x$ is $\Ga$--equivariant, that is, $\mu_{\ga x}=\ga_*\mu_x$ for all $x\in X$ and $\ga\in \Ga$,
\item the measures $\mu_x$ are pairwise equivalent with Radon--Nikodym derivatives given by
\[\frac{d\mu_x}{d\mu_y}(\theta)=e^{-\delta B_\theta(x,y)}\]
for all $x,y\in X$ and $\theta\in \supp{\mu_x}$.
\end{enumerate}
\end{definition}

Recall the \textbf{critical exponent of $\Ga$} is defined to be
\[\delta(\Ga):=\inf\{s\,:\,\sum_{\ga\in \Ga}e^{-sd({p},\ga {p})}<\infty\},\]
where ${p}\in X$ is any chosen basepoint. By Patterson's construction \cite{Patterson76} (see \cite[Lemma 2.2]{Knieper97}), there always exists a $\delta(\Ga)$--conformal density on the Hadamard space $X$ provided the discrete group $\Gamma<\op{Isom}(X)$ has $\delta(\Gamma)>0$. 

\subsection{The Natural flow}\label{sec:natural flow}

Given a $\delta$--conformal density $\{\mu_x\}_{x\in X}$ on $X$, we define $\overline f\colon X\rightarrow \mathbb R$ given by $\overline f(x)=-\log \norm{\mu_x}$, where $\norm{\mu_x}$ denotes the total mass of $\mu_x$ defined by
\[\norm{\mu_x}=\int_{\partial_\infty X}d\mu_x=\int_{\partial_\infty X}e^{-\delta B_\theta(x)}d\mu_{p}(\theta).\]
It follows from dominated convergence that $\overline f$ is also $C^2$. As $\mu_x$ is equivariant, $\bar{f}$ is $\Ga$--invariant, and thus descends to a $C^2$--function on $M$ which we denote by $f$.

Now we consider the gradient vector field on $M$ given by $\mathfrak X:=\nabla f$ which is also the quotient field of the $\Ga$--invariant gradient vector field $\nabla \overline f$ on $X$. A direct computation shows that
\begin{equation}\label{eq:df}
d{\overline f}(x)=\frac{\delta\cdot \int_{\pa_\infty X}dB_\theta(x) d\mu_x(\theta)}{\norm{\mu_x}},
\end{equation}
or equivalently, 
\[\nabla {\overline f}(x)=\frac{\delta\cdot \int_{\pa_\infty X}\nabla B_\theta(x) d\mu_x(\theta)}{\norm{\mu_x}}.\]
Additionally, we have
\begin{align}\label{eq:Hessian f}
\nabla d{\overline f}(x)&=\frac{\delta\cdot \int_{\pa_\infty X}\left( \nabla dB_\theta-\delta\cdot dB_\theta\otimes dB_\theta\right)(x) d\mu_x(\theta)}{\norm{\mu_x}}+\left(d{\overline f}\otimes d{\overline f}\right)(x).
\end{align}
In what follows, we abuse notation and do not distinguish between $f$ with $\overline f$, and $x$ with its lift since $\overline f$ is $\Ga$--invariant.

\begin{lemma}
The gradient vector field $\mathfrak X$ is complete and $C^1$--smooth.
\end{lemma}

\begin{proof}
As the Busemann function is $C^2$, $\mathfrak X$ is $C^1$--smooth. Since $M$ is complete, it suffices to show the flow does not blow up in finite time $t$. However, finite time blow-up cannot occur because the norm of $\mathfrak X$ is uniformly bounded. Indeed, for any $x\in M$,
\[\norm{\mathfrak X}(x)=\norm{\nabla f}(x)\leq \frac{\delta\cdot \int_{\pa_\infty X}\norm{\nabla B_\theta}(x) d\mu_x(\theta)}{\norm{\mu_x}}= \delta.\]
\end{proof}

We define a global gradient $C^1$--flow $\phi\colon \R\times M\rightarrow M$ along integral curves of $\mathfrak X$. For $t\in \R$, we denote $\phi_t(\cdot):=\phi(t,\cdot)$ and recall that $\phi_t$ is a $C^1$--self-diffeomorphism of $M$ that is homotopic to the identity. Note that if $z$ is a singular point of $\mf{X}$ (i.e. $\mf{X}(z)=0$), then $\phi_t(z)=z$ for all $t\in \R$. From equation \eqref{eq:df} we see that the set of singular points (or the fixed point set of the flow) coincides exactly with the fixed point set of the natural map constructed by Besson--Courtois--Gallot. If $z$ is a regular point (i.e. $\mf{X}(z)\neq0$), then $\phi_t(z)$ is a regular point for all $t\in \R$.

\subsection{Proof of Theorem \ref{prop:phi_contracts}} 

\begin{definition}
Let $M,N$ be smooth Riemannian manifolds,  $F\colon M\rightarrow N$ a $C^1$--smooth map, and $0 < k\leq \min\{\dim(M), \dim(N)\}$. We define the \textbf{$k$--Jacobian} function as
\[\op{Jac}_k F(x):=\sup \norm{dF(e_1)\wedge...\wedge dF(e_k)},\]
where the supremum is taken over all orthonormal $k$--frame $\{e_1,...,e_k\}\subset T_xM$, and the norm is taken with respect to the Riemannian metric of $N$ at $F(x)$.
\end{definition}

\begin{definition}\label{def:k-trace} 
Let $A$ be an $n$--dimensional symmetric bilinear form, and $\lambda_1\leq...\leq \lambda_n$ be its eigenvalues with respect to the given metric. For each integer $1\leq k\leq n$, we define the \textbf{$k$--trace of $A$} to be
\[ \op{tr}_k(A)=\sum_{i=1}^k \lambda_i, \]
or equivalently, $\op{tr}_k(A)=\inf\{\op{tr}(A|_{V^k}):V^k \text{ is any $k$--dimensional subspace}\}$. 
\end{definition}

We observe that $\op{tr}_k$ is superadditive on symmetric forms:  $\op{tr}_k(A+B)\geq \tr_k(A)+\tr_k(B)$.

\begin{lemma}\label{lem:volume_rate_trace}
Suppose $\ga\colon \R\rightarrow M$ is an integral curve based at $p=\ga(0)$, and that $Y_1,...,Y_k\in T_pM$ is any $k$--frame at p. If $Y_i(t):=(\phi_t)_*(Y_i)\in T_{\ga(t)}M$ and $V(t)\subset T_{\ga(t)}M$ is the subspace spanned by $\{Y_i(t)\}_{i=1}^k$, then $\norm{Y_1(t)\wedge...\wedge Y_k(t)}'=\left(\op{tr} \nabla d f|_{V(t)}\right)\norm{Y_1(t)\wedge...\wedge Y_k(t)}$.
\end{lemma}

\begin{proof} 
Pick an orthonormal $k$--frame $\{e_i(t): 1\leq i\leq k\}$ on each $V(t)$ so that initially it has the same orientation with $\{Y_1,...,Y_k\}$. Extend $Y_i$ on a neighborhood of $\ga(t)$ such that $[Y_i, \mathfrak X]=0$ for all $i$, this can be done for example by first choosing curves $\eta_i(s)$ through $p$ tangent to $Y_i$ and then pushing forward by the diffeomorphism $\phi_t$, the resulting vector field $(\phi_t)_*(\eta_i'(s))$ will then be an extension of $Y_i$ which commutes with $\mathfrak X$. Now we compute
\begin{align*}
&\norm{Y_1(t)\wedge...\wedge Y_k(t)}'=\Big(Y_1^*(t)\wedge...\wedge Y_k^*(t)(e_1(t),...,e_k(t))\Big)'\\
&=\Big(\nabla_{\mathfrak X}(Y_1\wedge...\wedge Y_k)\Big)^*(e_1,...,e_k)+\sum_{i=1}^k \Big(Y_1^*\wedge...\wedge Y_k^*(e_1,...,\nabla_{\mathfrak X} e_i,...,e_k)\Big)\\
&=\sum_{i=1}^k\Big(Y_1\wedge...\wedge \nabla_{\mathfrak X} Y_i\wedge...\wedge Y_k\Big)^*(e_1,...,e_k).
\end{align*}
In the above equations $*$ indicates the dual tensor with respect to the Riemannian inner product and note that $\nabla_\mathfrak X Y_i^*=(\nabla_\mathfrak X Y_i)^*$, and the last equality follows from the fact that $\langle\nabla_\mathfrak X e_i,e_i\rangle=0$. Since $[Y_i, \mathfrak X]=0$, we have $\nabla_\mathfrak X Y_i=\nabla _{Y_i} \mathfrak X$. Now if we denote $a_i(t)$ the $i$-th component in the basis $\{Y_i(t)\}_{i=1}^k$ of $V(t)$, then the last expression equals
\[\left(\sum_{i=1}^k a_i\right)\norm{Y_1\wedge...\wedge Y_k}(t).\]
If we denote $p(t)$ the projection from $T_{\ga(t)} M$ to $V(t)$, then $\sum_{i=1}^k a_i$ is exactly $\tr(p\circ \nabla \mathfrak X|_{V(t)})$, or equivalently, when replacing $\nabla \mathfrak X$ by the corresponding symmetric bilinear form $\nabla df$, it is the trace after restriction to $V(t)$. Hence the lemma follows.
\end{proof}

\begin{proof}[Proof of Theorem \ref{prop:phi_contracts}] 
For any $x\in M$ and $t_0>0$, denote $\{Y_1,...,Y_k\}\subset T_x M$ any orthonormal $k$--frame that spans a $k$--dimensional subspace $V\subset T_x M$. Apply Lemma \ref{lem:volume_rate_trace} and using the same notation there, we have for all $t\in [0,t_0]$.
\[\norm{Y_1(t)\wedge...\wedge Y_k(t)}'\geq  \epsilon_k\norm{Y_1(t)\wedge...\wedge Y_k(t)}.\]
Thus by the Gr\"onwall's inequality, we have
\[\norm{Y_1(t_0)\wedge...\wedge Y_k(t_0)}\geq e^{\epsilon_k t_0}\norm{Y_1(0)\wedge...\wedge Y_k(0)}=e^{\epsilon_k t_0}.\]
Thus, at the point $\phi_{t_0}(x)$, the map $\phi_{-t_0}$ (which is the inverse of $\phi_{t_0}$) satisfies $\op{Jac}_{k}\phi_{-t_0}\leq e^{-\epsilon_k t_0}$. Since $t_0$, $x$ are arbitrary, the theorem follows.
\end{proof}

\subsection{Proof of Corollary \ref{T:MainJac}}

We now prove Corollary \ref{T:MainJac}.

\begin{proof}[Proof of Corollary \ref{T:MainJac}]
We are restricting Theorem \ref{prop:phi_contracts} to manifolds whose sectional curvature is bounded above by $-1$ and setting $f(x)=-\log||\mu_x||$ as in \S \ref{sec:natural flow}. For each $x\in X$ and $\theta\in \partial_\infty X$, we choose the orthonormal frame $\{e_1,...,e_n\}$ on $T_xX$ such that $e_1=v_{x\theta}$ and $e_2,...,e_n$ be the eigenvectors corresponding to the second fundamental form of the horosphere passing through $x$ and $\theta$. Then the frame diagonalizes both $\nabla dB_\theta(x)$ and $dB_\theta\otimes dB_\theta (x)$, and it follows that $\left(\nabla dB_\theta-\delta\cdot dB_\theta\otimes dB_\theta\right)(x)$ equals $\op{diag}(-\delta, \kappa_1,..., \kappa_{n-1})$ where $\kappa_1, ..., \kappa_{n-1}$ are the principal curvatures of the horosphere, and under the curvature assumption of $X$, they are at least $1$. Therefore, in view of inequality \eqref{eq:Hessian f} and the fact that $df\otimes df$ is positive semi-definite, we have by the superadditivity of $\tr_k$ that
\begin{align*}
\tr_k\nabla df(x)&\geq \delta\cdot \int_{\partial_\infty X}\tr_k\left(\nabla dB_\theta -\delta\cdot dB_\theta \otimes dB_\theta \right)(x)\frac{d\mu_x}{||\mu_x||}\\
&\geq \delta (-\delta +k-1).
\end{align*}
The corollary then follows from Theorem \ref{prop:phi_contracts}.
\end{proof}

\subsection{Proof of Theorem \ref{Thm:DeyKapConj}} 

The proof is inspired by \cite[Prop 21]{Dey-Kapovich}. The improvement arises from using the natural flow instead of the natural map. The key observation is that under the assumption $\delta<2k$, the natural flow $\phi_{-t}$  infinitesimally contracts every holomorphic $k$--plane in the tangent bundle of $M$.

\begin{lem}\label{lem:holomorphic_contract}
Let $X = \mathbf{H}_\C^n$ be complex hyperbolic $n$--space. Following the same notations in \S \ref{sec:gradient_flow}, if $x\in M$, and $V\subset T_xM$ is any holomorphic $k$--plane (i.e. $J(V)=V$), then $\tr\nabla d f|_V \geq \delta(2k-\delta)$.
\end{lem}

\begin{proof}
According equation \eqref{eq:Hessian f}, and note $df\otimes df$ is positive semi-definite, we have
\begin{align}\label{eq:1}
\tr\nabla d f|_V\geq \delta\cdot \int_{\pa_\infty X}\tr\left( \nabla dB_\theta-\delta\cdot dB_\theta\otimes dB_\theta\right)|_{V}d\overline{\mu}_x(\theta),
\end{align}
where $d\overline{\mu}_x$ is the normalized Patterson--Sullivan measure at $x$. If we denote the symmetric bilinear forms $K_\theta=\nabla dB_\theta$, $H_\theta=dB_\theta\otimes dB_\theta$ and $L_\theta=K_\theta-\delta H_\theta$, then in complex hyperbolic spaces, we have $K_\theta+J^{-1}K_\theta J=2I$, and $H_\theta+J^{-1}H_\theta J=\op{diag}(1, 1, 0, \cdots, 0)$, after a suitable choice of orthonormal basis on $T_xM$ such that $e_1$ is in the direction $v_{x\theta}$ and $e_2=Je_1$. Thus, we have $L_\theta+J^{-1}L_\theta J=\op{diag}(2-\delta, 2-\delta, 2, \cdots, 2)$ under the same choice of orthonormal basis. On the other hand, we observe from $J(V)=V$  so that
\begin{align*}
2\tr L_\theta|_V&=\tr L_\theta|_V+\tr L_\theta|_{J(V)} =\tr (L_\theta + J^{-1}L_\theta J)|_V\\
&\geq \tr_{2k}(L_\theta + J^{-1}L_\theta J) =2(2k-\delta)
\end{align*}
The lemma then follows from inequality \eqref{eq:1}.
\end{proof}

\begin{proof}[Proof of Theorem \ref{Thm:DeyKapConj}]
Assume there exists a compact, complex subvariety $N\subset M$ of complex dimension $k$. As $\delta<2k$, by Lemma \ref{lem:volume_rate_trace} and Lemma \ref{lem:holomorphic_contract}, the flow $\phi_{-t}$ contracts almost all tangent planes (except possibly at singular points) of $N$ for sufficiently small $t>0$. If $\omega$ is the K\"ahler form on $M$, then $\frac{\omega^k}{k!}$ is a calibration of $N$ due to the Wirtinger's inequality \cite[\S 1.8.2]{Federer}. Hence $N$ is a volume minimizer in its homology class which contradicts that the flow $\phi_{-t}$ contracts.
\end{proof}

\begin{remark}
In \cite{Berger}, Berger showed the same minimizing fact for every quaternionic subvariety of a quaternionic K\"ahler manifold. Hence a similar theorem should hold for $X=\mathbf{H}_{\mathbb H}^n$.
\end{remark}

\subsection{Proof of Theorem \ref{T:MainIso} and Corollary \ref{Cor:Cheeger}} 

\begin{proof}[Proof of Theorem \ref{T:MainIso}]
Let $\mathcal A$ be the set of all compact smoothly immersed $(k+1)$--submanifolds whose boundary is $B$. Set $A_0=\inf\{\vol_{k+1}(\Omega):\Omega\in \mathcal A\}$. If $k+1=n$, then $A_0$ is simply achieved by a compact smooth region $\Omega_0$ that $B$ encloses. If $k+1<n$, then for any $\epsilon>0$, there exists $\Omega_0\in \mathcal A$ such that $\vol_{k+1}(\Omega_0)< A_0+\epsilon$. We apply our natural flow $\phi_{-t}:=F(t)$ on $\Omega_0$. Denote $\Omega_t=F_t(\Omega_0)$ and $B_t=F_t(B)$. Then for each $t_0>0$, the image 
\[ I_{t_0}=\bigcup_{0\leq t\leq t_0}F_t(B) \] forms a smooth (possibly degenerate) $(k+1)$--submanifold whose boundary is $B\cup B_{t_0}$. If we glue $I_{t_0}$ with $\Omega_{t_0}$ along the common boundary $B_{t_0}$, then we obtain a piecewise smooth submanifold in $M$, which after a small perturbation from the standard smoothing procedure, lies in $\mathcal A$, provided $k+1<n$. Thus, 
\begin{equation}\label{eq:triangle}
\vol_{k+1}(I_{t_0})+\vol_{k+1}(\Omega_{t_0})\geq A_0-\epsilon
\end{equation}
assuming the perturbation is small. In the case $k+1=n$, we simply have 
\[\vol_{n}(I_{t_0})+\vol_{n}(\Omega_{t_0})\geq \vol_n(\Omega_0)= A_0\]
hence \eqref{eq:triangle} also holds.
By \cref{T:MainJac}, we have 
\begin{equation}\label{eq:contracting}
\vol_{k+1}(\Omega_{t_0})\leq e^{-\delta(k-\delta)t_0}\vol_{k+1}(\Omega_0)\leq e^{-\delta(k-\delta)t_0}(A_0+\epsilon).
\end{equation}
Next, we estimate $\vol_{k+1}(I_{t_0})$. Let $G\colon B\times [0,t_0]\rightarrow M$ be the smooth map given by $G(x,t)=F_t(x)$. Then we can estimate the partial derivatives
\[||\frac{\partial G}{\partial t}||\leq ||\mathfrak X||\leq \delta,\]
and that 
\[||\frac{\partial G}{\partial x}||\leq ||\Jac_{k} F_{t_0}|| \leq e^{-\delta(k-1-\delta)t_0}.\]
Thus we can estimate
\begin{align*}
\vol_{k+1}(I_{t_0})&=\int_{B\times [0,t_0]}\Jac G\; dx dt \leq \int_{B\times [0,t_0]}||\frac{\partial G}{\partial x}\wedge \frac{\partial G}{\partial t}||\; dx dt\\
&\leq \int_{B\times [0,t_0]} e^{-\delta(k-1-\delta)t_0}\cdot \delta \;dxdt =e^{-\delta(k-1-\delta)t_0}\cdot \delta \cdot t_0\cdot \vol_k(B)
\end{align*}
Now combining with inequalities \eqref{eq:triangle} and \eqref{eq:contracting}, we have
\[(1-e^{-\delta(k-\delta)t_0})A_0\leq \epsilon+\epsilon \cdot e^{-\delta(k-\delta)t_0}+e^{-\delta(k-1-\delta)t_0}\cdot \delta\cdot t_0\cdot \vol_k(B).\]
Since $t_0$ and $\epsilon$ are arbitrary, by setting $t_0=\sqrt \epsilon$ and $\epsilon\rightarrow 0^+$, we obtain
\[A_0\leq \frac{1}{k-\delta}\vol_k(B).\]
The theorem then follows from the definition of $A_0$.
\end{proof}

Given a closed Riemannian manifold $M$ of dimension $n$, the \textbf{Cheeger constant} is defined as
\[h(M)=\inf_{A\subset M} \frac{\vol_{n-1}(\partial A)}{\vol_n(A)},\]
where the infimum is taking over all smooth regions $A\subset M$ (with $\partial A\subset M$ a smooth hypersurface) with no more than half of the total volume $\vol_n(M)$. The definition also extends to infinite volume complete manifolds, except that the infimum is now taking over all smooth region $A$ whose closure is compact. Cheeger \cite{Cheeger} proved that $h(M)\leq 2\sqrt \lambda_0$, where $\lambda_0$ is the infimum of the $L^2$-spectrum of the Laplacian on $M$. In \cite[Thm 7.1]{Buser}, Buser provides a lower bound (in the noncompact case)
\[h(M)\geq C(n)\cdot\frac{\lambda_0}{a},\]
where $-(n-1)\cdot a^2$ is the Ricci lower bound of $M$. If we apply our Theorem \ref{T:MainIso} to the case $k=n-1$, it gives a new explicit lower bound assuming a stronger curvature condition.

\begin{proof}[Proof of Corollary \ref{Cor:Cheeger}]
We may assume $\delta<n-1$, and in particular $M$ has infinite volume. According to the definition, for any $\eps>0$, there exists a compact smooth region $A\subset M$ with smooth boundary $\partial A$ whose volume satisfies
\begin{equation}\label{eq:cheeger1}
\frac{\vol_{n-1}(\partial A)}{\vol_n(A)}\leq h(M)+\eps.
\end{equation}
We notice that since $\partial A$ is a codimension one submanifold, $A$ is the unique compact region it encloses. If not, suppose $\partial A$ bounds another compact region $A'\subset M$. Take its opposite orientation and denote it by $A'_-$. Then after cancellation, $A\cup A'_-$ is a closed submanifold of $M$ of the same dimension. This is impossible since $M$ is connected and noncompact.
    
Thus, apply Theorem \ref{T:MainIso} to the case $k=n-1$, we have
\begin{equation}\label{eq:cheeger2}
\vol_{n}(A)\leq \frac{1}{n-1-\delta}\vol_{n-1}(\partial A)+\eps'
\end{equation}
for any $\eps'>0$. Since both $\eps, \eps'$ are arbitrary, we obtain from \eqref{eq:cheeger1} \eqref{eq:cheeger2} that $n-1-\delta\leq h(M)$.
\end{proof}

\begin{remark} 
Combining with Cheeger's inequality, we have
\[\lambda_0\geq \frac{(n-1-\delta)^2}{4}.\]
Consider the lower volume growth entropy of $M$, 
\[ h_M=\liminf \frac{\log \vol B(x,R)}{R} \] 
for any $x\in M$. (Observe that $h_M\leq h$ the volume growth entropy of $\til{M}$.) Combining the above corollary with Brooks' inequality \cite{Brooks}, we have $h_M\geq h(M)\geq (n-1-\delta)$.
\end{remark}
 
\section{Homological dimension}\label{sec:homdim}

\subsection{Homological and Cohomological dimension}\label{sec:hom_prelim} 


Given a commutative ring $R$ with identity, a group $\Gamma$ has \textbf{$R$--cohomological dimension $k$}, denoted $\op{cd}_R(\Gamma)=k$, if $k$ is the smallest integer for which there exists a resolution by projective $R \Gamma$--modules of the form 
\[ 0 \rightarrow P_k \rightarrow P_{k-1} \rightarrow \cdots \rightarrow P_0 \rightarrow R \rightarrow 0. \] 
Similarly, a group $\Gamma$ is said to have \textbf{$R$--homological dimension $k$} over $R$, denoted $\op{hd}_R(\Gamma)=k$, if $k$ is the smallest integer for which there exists a resolution by flat $R\Gamma$--modules of the form $0 \rightarrow F_k \rightarrow F_{k-1} \rightarrow \cdots \rightarrow F_0 \rightarrow R \rightarrow 0$. In other words, the (co)homological dimension of $\Gamma$ is just the (projective) flat dimension of the $R\Gamma$--module $R$. The following will be useful for us.

\begin{theorem}(\cite[Proposition 4.1 and \S 4.4]{Bieri76}). 
For a group $\Gamma$ and a commutative ring with identity, we have
\begin{align*}
&\operatorname{cd}_R(\Gamma)=\sup \left\{n: \text{ there exists an } R \Gamma\text{-module } V \text { such that } H^n(\Gamma ; V) \neq 0\right\}, \\
&\operatorname{hd}_R(\Gamma)=\sup \left\{n: \text{ there exists an } R\Gamma\text{-module } V \text { such that } H_n(\Gamma ; V) \neq 0\right\}.
\end{align*}
\end{theorem}

Note that as every projective module is flat, we always have $\op{cd}_R(\Gamma)\geq \op{hd}_R(\Gamma)$. Also, if $\mathbb Z$ (as a trivial $\mathbb Z\Ga$--module) has a projective resolution of length $k$, then tensoring with $R\Ga$ gives rise to a projective resolution (as $R\Ga$--modules) of $R$, whose length is at most $k$. Thus, we always have $\cd_R(\Gamma)\leq \cd(\Gamma)$ (recall $\cd(\Gamma)$ means $\cd_\Z(\Gamma)$).

For the statement and proofs of some of our results, we will need to both ``spacify" and ``relativize'' the above notions of (co)homology. Let $X$ be a CW--complex. For any sheaf $\bb{E}$ of modules on $X$ the sheaf cohomology $H^*(X,\bb{E})$ is defined to be the image of the corresponding morphism to abelian groups of the right derived functor of the functor $\bb{E}\to \Gamma(X,\bb{E})$ where $\Gamma(X,\bb{E})$ are global sections of $\bb{E}$. Consider a special case where $X$ be the universal cover of a CW--complex $\Omega$ and $\bb{E}$ is the product sheaf over $X$ with fibers $V_R$ where $V_R$ is the $R\Gamma$--module $V$ regarded as an $R$--module. We may think of the sheaf $\bb{E}$ as the sheaf of local (horizontal) sections of the product bundle $X \times V_R \rightarrow X$. In fact, we will identify this bundle with its sheaf of (horizontal) sections, and thus also denote it by $\bb{E}$. The group $\Gamma=\pi_1(\Omega)$ acts on this sheaf/bundle diagonally:
\[ \gamma \cdot(x, v)=(\gamma(x), \gamma \cdot v), \quad \gamma \in \Gamma. \]
The quotient under $\Gamma$ gives a bundle $\bb{V}\to \Omega$ which again can be identified with its sheaf of local horizontal sections and gives rise to the corresponding sheaf cohomology $H^*(\Omega,\bb{V})$. Moreover, for a subcomplex $C\subset \Omega$ we may define the relative sheaf cohomology of the pair $(\Omega,C)$  (see \cite{JohnsonMillson87}). We will denote this by $H^*(\Omega,C;\bb{V})$ for the sheaf $\bb{V}$ as above.
By taking the $\mathrm{Hom}$ functor on the associated cochain complex, we can also define the relative homology groups $H_*(\Omega,C;\bb{V})$.

Consider the special case when $\Omega=K(\Gamma,1)$ is the Eilenberg--MacLane space for $\Gamma$, and $C=\bigsqcup_i C_i$ where $\set{C_i}_{i\in I}$ is a disjoint family of embedded subcomplexes of $\Omega$ with each $C_i$ homeomorphic to an Eilenberg--Maclane space $K(\Pi_i,1)$ for a subgroup $\Pi_i<\Gamma$. By \cite[\S 1.5]{BieriEckmann78}, for the collection $\set{\Pi_i}_{i\in I}$ we have the following natural isomorphisms to the relative group (co)homology,
\[ H^*(\Gamma, \Pi ; V) \cong H^*(\Omega, C ; \mathbb{V}), \quad H_*(\Gamma, \Pi ; V) \cong H_*(\Omega, C ; \mathbb{V}). \]

In the case of $R=\mathbb{Z}$, we will omit the subscript from the notation for the (co)homological dimension. The relative homological and cohomological dimensions are tightly related.

\begin{remark}
Locally constant and locally free abelian sheaves can be identified with flat vector bundles, i.e. vector bundles carrying a connection whose associated curvature tensor vanishes, or equivalently, those whose structure group reduces to a discrete subgroup $\pi<\Gamma$. This follows from the fact that the classifying maps of flat bundles factor through the corresponding $K(\pi,1)$. We therefore follow the terminology of \cite{Kapovich} in calling $\mathbb{V}$ the flat bundle associated to the $R\Gamma$--module $V$. Notice also that this formulation is also equivalent to that of systems of local coefficients (see \cite[\S 4.3.4 and \S 5.3]{DavisKirk}).
\end{remark}

\begin{lemma}[{\cite[Lemma 2.8]{Kapovich}}]
For a group $\Gamma$ and collection of subgroups $\Pi$ as above,
\[ \mathrm{hd}_R(\Gamma, \Pi) \leq \mathrm{cd}_R(\Gamma, \Pi) \leq \mathrm{hd}_R(\Gamma, \Pi)+1. \]
\end{lemma}

\subsection{Critical index}

Let $X$ be a Hadamard space with $p \in X$, and $\Ga<\op{Isom}(X)$ be a discrete subgroup. Let $\mu_x$ be any conformal density of dimension $\delta_\mu$. Denote $\partial_\mu X\subset \partial_\infty X$ the support of $\mu$. For any $x\in X$, $\theta\in \partial_\infty X$, we denote the Hessian of the Busemann function $B_\theta(x)$ by $\nabla dB_{(x,\theta)}$, which always have a null space in the direction $v_{x\theta}$, and when restricted to its orthogonal complement is the same as the second fundamental form of the horosphere through $x$ and $\theta$.

\begin{definition}\label{D:CritIndex} 
Let $X$ be an $n$--dimensional Hadamard space, $\Ga<\op{Isom}(X)$ a torsion-free discrete subgroup, and $\mu_x$ be any conformal density of dimension $\delta_\mu$. We define the \textbf{critical index} of $\Ga$ associated to $\mu$ to be
\[j_X(\Gamma,\mu) = j_X(\mu):=\min\left(\{k\in \mathbb N: \inf_{(x,\theta)\in X\times \partial_\mu X} \op{tr}_k(\nabla dB_{(x,\theta)})>\delta_\mu\}\cup\set{n+1}\right).\]	
and the critical index of $\Ga$
\[ j_X(\Gamma)=\inf \set{j_X(\mu): \mu \text{ is a } \delta_\mu-\text{conformal density}}.\]
\end{definition}

\begin{remark}
In the case that $K_X\leq -1$, then $j_X(\Gamma)\leq \floor{\delta(\Gamma)}+2$ (see Corollary \ref{cor:neg_curved}). In many cases, $j_X(\Ga)$ is achieved when $\delta_\mu=\delta(\Ga)$, and the corresponding $\delta(\Ga)$--conformal density $\mu$ is known to be unique up to scale, hence $j_X(\mu)$ only depends on $\Ga$ and $j_X(\mu)=j_X(\Gamma)$.
\end{remark}

We now restate Theorem \ref{thm:main} from the introduction which we will prove in this section. Recall that $M_{\eps}$ is the $\eps$--thin part of $M$.

\begin{reptheorem}{thm:main}
If $X$ is a Hadamard space with bounded sectional curvatures, $\Gamma < \mathrm{Isom}(X)$ is a discrete, torsion-free subgroup with $M = X/\Gamma$, and $\mathbb{V}$ is a flat bundle over $M$, then for every $\epsilon > 0$, the homomorphism $i_k:H_k(M_{\eps};\mathbb V)\rightarrow H_k(M;\mathbb V)$ induced by inclusion is surjective whenever $k\geq  j_X(\Gamma)$.
\end{reptheorem}

\begin{remark}\label{rem:long_exact}
In particular, if for some choice of $\eps>0$, $M_{\eps}$ is empty or otherwise $H_k( M_{\eps};\mathbb{V})=0$, then $H_{k}(M;\mathbb{V})=0$  for any $k \geq j_X(\Gamma)$. By considering the long exact sequence
\[ \begin{tikzcd} \cdots\ar[r,"\partial_{k+1}"]&H_k(M_{\eps};\bb{V})\ar[r,"i_k"]&H_k(M;\bb{V})\ar[r,"q_k"]&H_k(M,M_{\eps};\bb{V})\ar[r,"\partial_k"]&\cdots \end{tikzcd} \]
We see that the surjectivity of the map $i_k$ is equivalent to either the vanishing of the map $q_k$ or the injectivity of $\partial_k$. While we do not in general obtain that $H_k(M,M_{\eps};\V)=0$, we do obtain this in an important special case (see Theorem \ref{thm:relative-vanishing}).
\end{remark} 

\begin{corollary}\label{cor:hd_bound}
Given a Hadamard space $X$ whose sectional curvature is bounded from below, and any discrete, torsion-free subgroup $\Gamma<\op{Isom}(X)$. If the injectivity radius of $X/\Ga$ has a uniform positive lower bound, then the homological dimension satisfies $\op{hd}_R(\Ga)\leq j_X(\Ga)-1$.
\end{corollary}

\begin{proof}
As the injectivity radius of $M=X/\Ga$ is bounded below by $\eps_0$, for $\eps<\eps_0$ we have $H_k(M_{\eps},\bb{V})=0$. Hence Theorem \ref{thm:main} implies that $H_k(M,\bb{V})=0$ for all $k\geq j_X(\Ga)$ and flat bundles $\bb{V}$. Thus $\mathrm{hd}_R(\Gamma) \leq j_X(\Gamma)-1$. 
\end{proof}

\begin{corollary}\label{cor:neg_curved}
If $X$ has pinched negative curvature $-b^2\leq K_X\leq -1$ and $\Gamma < \mathrm{Isom}(X)$ is a discrete, torsion-free subgroup with no parabolics, then $\op{hd}_R(\Gamma)\leq j_X(\Gamma)-1\leq \delta(\Gamma)+1$.
\end{corollary}

\begin{remark}
Corollary \ref{cor:neg_curved} fails when the group $\Ga$ is allowed to have parabolics. For example, consider the real hyperbolic $n$--manifold $X/\Ga$ where $\Ga\cong\Z^{n-1}$ is a parabolic subgroup acting cocompactly on a horosphere, as a lattice in $\R^{n-1}$. In this case the exponential distortion of the horospheres gives $\delta(\Ga)=\frac{n-1}{2}$, even though the limit set is one point. However, the homological dimension is $\op{hd}_R(X/\Ga)=n-1$. Generalizations of this phenomena occur when $\Ga$ has a $\Z^{k}$ parabolic subgroup when $k$ is large enough.
\end{remark}

\begin{proof}[Proof of Corollary \ref{cor:neg_curved}]
We observe that since $\Gamma$ has no parabolic or torsion elements, it consists entirely of hyperbolic elements. By the Margulis Lemma, for $\eps<\eps_0$ where $\eps_0$ is the Margulis constant, $M_{\eps}$ is a disjoint union of a countable number of tubes $\coprod_i T_i$, where each component tube $T_i$  is homeomorphic to $S^1\times D^{n-1}$. Hence for all $m\geq 0$, 
\[ H_m(M_{\eps};\bb{V})=H_m\left(\coprod_i T_i;\bb{V}\right)\cong\bigoplus_i H_m(T_i;\bb{V}) \] 
which vanishes for $m\geq 2$ and any flat bundle $\bb{V}$ since each $T_i$ is homotopy equivalent to $S^1$. Hence by \cref{thm:main}, $\op{hd}_R(\Gamma)\leq j_X(\Gamma)-1$. We claim that $j_X(\Ga)\leq \floor{\delta(\Ga)}+2$. Recall that $\delta_\mu=\delta(\Gamma)$ if we choose $\mu$ to be the Patterson--Sullivan measure. Since of the curvature upper bound, for any $x\in M$ and $\theta\in \partial_\infty X$, $\nabla dB_{(x,\theta)}$ has one zero eigenvalue and all others at least $1$. Hence when $\floor{\delta(\Ga)}+2\leq n$, then \[ \tr_k(\nabla dB_{(x,\theta)})\geq\floor{\delta(\Ga)}+1>\delta(\Ga),\]
for $k\geq \floor{\delta(\Ga)}+2$ and the claim follows. If $\floor{\delta(\Ga)}+2> n$, then $j_X(\Ga)\leq n+1\leq \floor{\delta(\Ga)}+2$, and the claim also holds. As $j_X(\Ga)\geq 2$ the corollary now follows from Theorem \ref{thm:main}.
\end{proof}

Corollary \ref{cor:hd_main2} is the combination of our previous two corollaries. 

\begin{remark}
Note that there is no interesting lower bound for $\op{hd}_R(\Gamma)$ in terms of $\delta(\Gamma)$ as one can build finitely generated free subgroups of $\op{Isom}(\mathbf{H}_\R^n)$ via the generalized Schottky procedure such that $\delta(\Gamma)$ is large, but $\op{hd}(\Gamma)=1$ for any free group. Although when $n\geq 4$ is even, among geometrically finite representations of a subgroup of a finite volume hyperbolic 3--manifold group, $\delta(\Gamma)$ cannot be made arbitrarily close to $n-1$ (\cite[Cor 1.4]{Bowen15}).
\end{remark}

\subsection{Gromov's Mass Vanishing Theorem}

To prove Theorem \ref{thm:main}, we will need a generalization of Gromov's mass vanishing theorem (\cite{Gromov07}) to nonpositive curvature. Following \cite{Kapovich}, we introduce the notion of volume for a relative homology class. Let $X$ be either a simplicial complex or a Riemannian manifold, and let $Y\subset X$ be either a subcomplex or a closed submanifold with piecewise-smooth boundary. For any piecewise smooth $k$--simplex $\sigma\colon \Delta^k\rightarrow X$, we define the \textbf{relative volume} to be
\[\vol(\sigma|_{X\backslash Y})=\int_{\sigma^{-1}(X\backslash Y)\subset \Delta^k}\Jac_k\sigma dV_{\Delta},\]
where $dV_{\Delta}$ is the volume form on the Euclidean $k$--simplex $\Delta^k$. For $\alpha\in H_k(X,Y,\bb V)$, we define the \textbf{volume of $\alpha$} to be
 \[\vol(\alpha):=\inf \{\vol(c)~|~c\textit{ is a piecewise smooth chain representing }\alpha\},\]
where
\[\vol(c)=\sum_{i=1}^\ell \vol(\sigma_i|_{X\backslash Y}),\]
if $c=\sum_{i=1}^\ell v_i\otimes \sigma_i$ for $v_i\in \bb V$, and $\sigma_i\colon \Delta_k\rightarrow X$ a piecewise smooth simplex in $X$. Note that the volume of a cycle does not take into account the coefficients $v_i$. 

\begin{theorem}[Mass Vanishing Theorem]\label{thm:small_vanishing}
Let $X$ be an $n$--dimensional Hadamard space whose sectional curvature is uniformly bounded from below by $-a^2$, $\Ga<\op{Isom}(X)$ a discrete, torsion-free subgroup and $V$ a $R\Ga$--module associated to a flat bundle $\V$ over $M=X/\Ga$. For any $\epsilon>0$, there exists $\Theta=\Theta(\epsilon,n,a)>0$ such that any relative homology class $\alpha\in H_k(M,M_{\epsilon};\bb V)$ is either trivial under the map (induced by the inclusions)
 \[i_*:H_k(M,M_{\epsilon};\bb V)\rightarrow H_k(M,M_{2\epsilon};\bb V),\]
or $\mathrm{Vol}(\alpha) \geq \Theta$.
\end{theorem}

\cref{thm:small_vanishing} qualitatively generalizes \cite[Thm 7.1]{Kapovich}, which treats the constant curvature case of Theorem \ref{thm:small_vanishing}. The proof of \cref{thm:small_vanishing} is carried out in \cref{app:GromovTheorem} and closely follows that of \cite[Thm 7.1]{Kapovich} with certain necessary modifications. Most notably, we need to change the geodesic filling argument and estimates to those of a barycentric filling.

\subsection{Proof of Theorem \ref{thm:main}}

Using Theorem \ref{thm:small_vanishing} and  Theorem \ref{prop:phi_contracts}, we now prove Theorem \ref{thm:main}.

\begin{proof}
Since the homological dimension of $M$ is always at most $n$, we may assume that we are in the case where $j_X(\Gamma)< n+1$. By Remark \ref{rem:long_exact}, the surjectivity of $i_k$ is equivalent to the vanishing of the map $q=q_k$. We will show this. Consider any nontrivial class $\alpha\in H_k(M;\bb{V})$ and a representative chain $c$ of $\alpha$. We recall that the map $q$ is induced by the map at the chain level which we denote by the same letter $q$. Writing the singular chain $c=\sum_i^{\ell} v_i\otimes\sigma_i $ for $\sigma_i\colon \Delta^k\to M$ and $v_i\in \bb V$, we have the push-forward $(\phi_{-t})_*c=\sum_i v_i\otimes \left(\phi_{-t}\of\sigma_i\right) $. 

When $k\geq j_X(\Ga)$ then by hypothesis there is some $\eps_0>0$ such that $\tr_k \nabla dB_{(x,\theta)}\geq \delta(\Ga)+\eps_0$ for each $x\in X$ and $\theta$ in the support of $\mu_x$. Hence $\tr_k \nabla df\geq \eps_0$ by Equation \eqref{eq:Hessian f}. Apply Theorem \ref{prop:phi_contracts}, for sufficiently large $t>0$, the volume of each simplex map  $\phi_{-t}\of\sigma_i$ is bounded by 
\[ \vol(\phi_{-t}\of\sigma_i)\leq \frac{\Theta}{2\ell}, \]
where $\Theta$ is the constant in Theorem \ref{thm:small_vanishing}. Hence the total volume of $(\phi_{-t})_*c$ is less than $\Theta$. The map $q$ is induced by inclusion of chains, and hence the volume of $q((\phi_{-t})_*c)$ is no more than that of $(\phi_{-t})_*c$. By Theorem \ref{thm:small_vanishing}, $i_* (q(\alpha))=0\in H_k(M,M_{2\eps};\V)$. This shows that the natural map $H_k(M;\V)\rightarrow H_k(M,M_{2\eps};\V)$ is the zero map for any $\eps>0$. Therefore the theorem follows.
\end{proof}

\subsection{Locally symmetric spaces}

We apply our Theorem \ref{thm:main} to the setting of locally symmetric spaces. For the present paper, we will only consider the case of real rank one, and reserve the discussion of the higher rank setting to the follow-up paper \cite{CMW26}. When $X$ is a rank one symmetric space, the isometry group acts transitively on the unit tangent bundle.  Hence the second fundamental forms of horospheres, $\nabla dB_{(x,\theta)}$, do not depend on $x$ and $\theta$. As a result, $j_X(\Gamma)$ is rather simple to compute/estimate. In fact, we are able to recover and complement the results of \cite{Kapovich} and \cite{CFM19}.

\begin{proof}[Proof of Corollary \ref{cor:rank one}] Observe that for rank one symmetric spaces and any pair $(x,\theta)\in X\times \pa_\infty X$, the bilinear forms $\nabla dB_{(x,\theta)}$ are conjugate to each other, hence they have the same $k$--trace. After choosing a diagonalized basis (similar to the proof of Corollary \ref{T:MainJac}), we may write $\nabla dB_{(x,\theta)}$ in each case respectively as:
\begin{enumerate}
\item $\op{diag}(0, 1, ..., 1)$ when $X$ is real hyperbolic,
\item $\op{diag}(0, 1,..,1, 2)$ when $X$ is complex hyperbolic,
\item $\op{diag}(0, 1,..,1, 2, 2, 2)$ when $X$ is quaternionic hyperbolic, and
\item $\op{diag}(0, 1^{(8)}, 2^{(7)})$ when $X$ is Cayley hyperbolic.
\end{enumerate}
By carefully adding up the $k$--trace of $\nabla dB_{(x,\theta)}$ in each case, the result follows directly from Corollary \ref{cor:neg_curved}.
\end{proof}

When $X$ is higher rank, $\nabla dB_{(x,\theta)}$ depends on $\theta$ (although it still does not depend on $x$), so the support of the conformal density becomes rather crucial when it comes to the computations and estimates of the critical index. In fact, the particular choice of $G$--orbit of $\theta$ is very delicate. It relies on several ingredients, including the extension of Patterson--Sullivan theory formulated by Benoist and Quint \cite{Benoist97,Quint-divergence02,Quint02,Quint03} (see also \cite{Link04}), together with an upper bound estimate of the indicator function given by Lee and Oh \cite{LeeOh22, Oh98,Oh02}, which is also partly hinted in \cite{FraczykLowe}. Again, this will be discussed more fully in \cite{CMW26}.

\subsection{Zero-thin homology at infinity}

In Theorem \ref{thm:main}, there remains the possibility that the subgroups arising from the thin part could possibly contribute a large (co)homological dimension. Note that in the case of  pinched negative curvature, this can only happen in the presence of cusps, of varying ranks, since the Margulis tubes always have fundamental group $\mathbb Z$ whose (co)homological dimension is one. However, in the case of general Hadamard space, or when the group is not finitely generated, it is possible to have thinner and thinner compact regions escaping to infinity that produce a large (co)homological dimension while contributing a small amount to the critical exponent. Here we present two cautionary examples that show our main theorem cannot be improved with respect to the homology of the thin part in its current generality. In the first example, $H_{n-1}(M;\V)=\lim_{\eps\to 0} i_*H_{n-1}(M_{\eps};\V)\neq 0$. In particular, some nonvanishing residual ``zero-thin'' homology can remain even in codimension one. 

\begin{example}\label{ex:lochness}
Consider an infinite genus hyperbolic surface $\Sigma$ whose topological type is that of the double Loch Ness monster which can be described as an infinite connect sum $\Sigma=\#_{i\in\Z} S_i$ where for each $i$, $S_i$ is a compactification of a doubly punctured torus with boundary circles $a_i$ and $b_i$ and $b_i$ is glued to $a_{i+1}$ for each $i$, and the resulting curve in $\Sigma$ will be labelled $c_i$. 

It is well known that such a surface carries a metric of constant curvature $-1$ such that the curves $c_i$ are geodesics whose lengths can be chosen to be any sufficiently small positive numbers. We choose these to be $\frac1i$. In this metric, the $\eps$--thin part, $\Sigma_\eps$, for sufficiently small $\eps>0$, is an infinite disjoint union of tubes centered around $c_i$ for all sufficiently large $\abs{i}$. Unlike the finite genus case, the curves $c_i$ all represent a nontrivial homology class, and they are all homologous to $[c_0]\in H_1(\Sigma;\Z)$. In particular, the intersection of the inclusion maps of the homology of $\Sigma_\eps$ is 
\[ \bigcap_{\eps\to 0} i_*H_1(\Sigma_\eps;\Z)=\inner{[c_0]}<H_1(\Sigma;\Z), \] 
and does not vanish. Taking the Riemannian product $M=\Sigma\times N$ of $\Sigma$ with a closed non-positively curved manifold  of dimension $n-2$ yields a non-positively curved $n$--manifold. The K\"unneth Theorem together with functoriality of the direct limit yields that 
\[ H_{n-1}(M,\Z)\supset  \bigcap_{\eps\to 0} i_*H_{n-1}(M_{\eps},\Z)\supset\inner{[c_0\times N]}\cong \Z. \] 
In particular, for $\Gamma=\pi_1(M)$, $H_{n-1}(\Gamma,\Z)\neq 0$. When $N$ is chosen to be a hyperbolic surface, then $M$ is covered by the rank $2$ symmetric space $\mathbf{H}_\R^2\times\mathbf{H}_\R^2$. We are not immediately aware of an example of this form in an irreducible higher rank symmetric space, but we suspect that they exist. We also remark that the corresponding locally finite homology does vanish in dimension $3$ for this example.
\end{example}

\begin{remark}
\cref{ex:lochness} in the case when $N$ is a point shows that even in constant curvature, and for any $\eps>0$, $M_{\eps}$ need not be either empty or consist solely of parabolic cusp regions. In the non-positively curved case, we see $M_{\eps}$ can be quite complicated.
\end{remark}

In the second example, we show that even when $M_{\eps}=\emptyset$ for sufficiently small $\eps>0$, we may still have nonvanishing homology in codimension $1$ for open manifolds.

\begin{example}\label{ex:Heintze}
(Heintze Construction \cite{Schroeder91}) Consider a geometrically finite, infinite volume, manifold $N$ with constant curvature $-1$ and a finite number of maximal rank ($n-1$) cusps. As each cusp is a topological product $T^{n-1}\times [0,\infty)$ on which the constant curvature metric has a warped product structure, we can modify the convex warping function to smoothly transition from $-1$ to $0$ curvature in finite time down each cusp (see e.g. \cite{ConnellSuarez-Serrato19} for explicit details). Let the resulting non-positively curved manifold be $M$. The parabolic torus cross-sections down each cusp of $M$ eventually become totally geodesic and the remainder of each end is metrically a product. Hence, the injectivity radius of this new metric is bounded below. Nevertheless, for any $\eps<\op{injrad}(M)$, the tori in the ends represent a nontrivial class of $H_{n-1}(M;\Z)=H_{n-1}(M,M_{\eps};\Z)$. (By the infinite volume condition none of them bound a finite $n$--chain.) This example shows that we necessarily have $j_X(\Gamma)\geq n$. (It is not hard to see that $j_X(\Gamma)=n+1$ as the cusp cross section is a totally geodesic flat torus and thus $\nabla dB_{(x,\theta)}$ vanishes in the cusp direction.)
\end{example}

\subsection{Vanishing of relative homology}

From the long exact sequence of homology for the pair $(M,M_{\eps})$, we see that $H_k(M,M_{\eps};\mathbb V)=0$ implies the surjectivity of the map $H_k(M_{\eps};\V)\rightarrow H_k(M;\V)$. Thus obtaining a vanishing on the relative homology in the corresponding range of degrees is potentially an improvement of our Theorem \ref{thm:main}. In fact, when the topology of $M_{\eps}$ is less complicated as in the real hyperbolic case, Kapovich \cite{Kapovich} proved the vanishing of relative homology for the pair $(M,P_\eps)$ where $P_\eps$ is the cuspidal region of the $M_{\eps}$ and  $\eps$ is smaller than the Margulis constant $\eps_0$. We now generalize his result.

For general non-positively curved manifolds, the unbounded regions of the thin part may behave wildly. We say that a (necessarily connected) subset $R\subset M$ is \textbf{cuspidal} if $R$ admits a deformation retract onto a closed subset $R_0\subset R$ which is the quotient, under the universal covering map $X\to M$, of a horoball in $X$ by a subgroup $\Gamma_0<\Gamma$ stabilizing this horoball in $X$.

\begin{theorem}\label{thm:relative-vanishing}
Let $X$ be a Hadamard space and $\Gamma < \mathrm{Isom}(X)$ a discrete, torsion-free subgroup with $M = X/\Gamma$, $\mathbb{V}$ is a flat bundle over $M$, $k\geq  j_X(\Gamma)$, and $\epsilon > 0$. 
\begin{enumerate}
\item If $X$ has pinched negative curvatures, and $\eps<\eps_0$, then $H_{k+1}(M,M_{\eps};\mathbb V)=0$ and $H_{k}(M,P_{\eps};\mathbb V)=0$ where $P_\eps$ is the cuspidal part of $M_{\eps}$.
\item Let $\set{\mc{C}_{\eps}^\alpha}_{\alpha\in \mc{A}}$ be the collection of connected components of  $M_{\eps}$. If for each $\alpha$, either $H_{k}(\mc{C}_{\eps}^\alpha)=0$ or $\mc{C}_\eps^\alpha$ is cuspidal, then $H_{k+1}(M,M_{\eps};\mathbb V)=0$ and $H_{k}(M,P_{\eps};\mathbb V)=0$ where $P_\eps$ is the union of all cuspidal components of $M_{\eps}$.
\end{enumerate}
\end{theorem}

\begin{proof}
We assume as usual $j_X(\Ga)\leq n$, and otherwise the result holds automatically. Since $j_X(\Ga)\geq 2$, in either case, the $k$-th homology vanishes on the components of $M_{\eps}$ not in $P_{\eps}$, we have $H_{k}(M_{\eps},P_{\eps};\V)=\til{H}_{k}(M_{\eps}/P_{\eps};\V)=0$.  (Note that $M_{\eps}/P_{\eps}$ is just the disjoint union of a point and the components with vanishing $k$ homology and so the restriction of $\V$ still makes sense on this space.)  From the long exact sequence
\[ \begin{tikzcd} \cdots\ar[r,"\partial_{s+1}"]& H_s(M_{\eps},P_{\eps};\bb{V})\ar[r,"i_s"] & H_s(M,P_{\eps};\bb{V})\ar[r,"q_s"] & H_s(M,M_{\eps};\bb{V})\ar[r,"\partial_s"] & \cdots \end{tikzcd} \]
for the triple $(M,M_{\eps},P_{\eps})$, we conclude that the map $q_s$ is surjective when $s=k+1$ and is injective when $s=k$. We claim that $q_s$ is the zero map if $s\geq k$. Suppose we have this, then together with the surjectivity of $q_{k+1}$ and the injectivity of $q_k$ our result follows.

We now prove the claim $q_s=0$ whenever $s\geq k$.

Suppose $\beta\in C_s(M,P_{\eps};\V)$ is a singular cycle representing an arbitrary class in $H_s(M,P_{\eps};\V)$. In the first case let $P^0_{\eps}=P_{\eps}$. In the second case, let $P^0_{\eps}$ represent the union of the horoball quotients, one for each cuspidal component $C_{\eps}^{\alpha}$, contained in $P_{\eps}$. Since $P^0_{\eps}$ is a deformation retract of $P_{\eps}$, there is a cycle $\beta_0\in C_s(M,P_{\eps}^0;\V)$ whose natural inclusion into $C_s(M,P_{\eps};\V)$ represents the same homology class as $[\beta]$. Flowing for time $t$ along the unique length-minimizing unit speed geodesic starting at each point of $P_\eps$ and exiting out the respective cusp end induces a one parameter family of homeomorphisms $\psi_t\colon P_{\eps}^0\to P_{\eps}^t$ where $P_{\eps}^t$ is the image of $\psi_t$ in $P_{\eps}$. Since in either case, $\psi_t$ is a deformation retract, there is a cycle $\beta_t\in C_s(M,P_{\eps}^t;\V)$ whose inclusion in $C_s(M,P_{\eps};\V)$ is homologous to $\beta_0$. Moreover, $\beta_t$ can be constructed from $\beta_0$ by applying the flow $\psi_t$ to each simplex of $\partial \beta_0$ and subdividing appropriately. 

In the first case, $\beta_0=\beta$, and set $\eps(t)=\op{injrad}(P_{\eps}^t)$. Note that by uniform negative curvature $\eps(t)\rightarrow 0$ when $t\rightarrow \infty$. Moreover, by pinched negative curvature we know that the $\vol_s(\beta_t-\beta)$ is uniformly bounded (with respect to $t$) by $C_0\cdot\vol_{s-1}(\partial \beta)$ where $C_0$ depends on the pinching constants. It follows that $\vol_s(\beta_t)\leq \vol_s(\beta)+C_0\cdot \vol_{s-1}(\partial \beta)$. Apply our natural flow $\phi_{-t}$ for time $t$, since $s\geq k\geq  j_X(\Ga)$, it contracts the $s$--volume, and for sufficiently long time $t=T$, we may assume $\vol_s(\phi_{-T}\circ \beta_t)\leq \Theta$ for any $t$, where $\Theta$ is the constant as in Theorem \ref{thm:small_vanishing} for the scale $\eps/2$. Now we make the choice of $t=t_0$ large enough such that $\eps(t_0)<\eps/2e^{\delta T}$. Note that $\phi_{-T}$ is $e^{\delta T}$--Lipschitz by Theorem \ref{prop:phi_contracts}, so $\phi_{-T}\circ \beta_{t_0}$ lies in $C_s(M,P_{\eps/2};\V)$. Apply Theorem \ref{thm:small_vanishing} (at the scale of $\eps/2$), we see that $\phi_{-T}\circ \beta_{t_0}$ is trivial in $H_s(M,M_{\eps};\V)$. Finally, we see that $\phi_{-T}\circ \beta_{t_0}$ is homologous to $\beta_{t_0}$ in $C_s(M,P_{\eps/2};\V)$ hence also in $C_s(M,M_{\eps};\V)$, and $\beta_{t_0}$ is homologous to $\beta$ in $C_s(M,P_{\eps};\V)$ hence also in $C_s(M,M_{\eps};\V)$. Hence $q_s([\beta])\in H_s(M,M_{\eps};\V)$ vanishes and so $q_s=0$.

In the second case, $P_{\eps}^t$ is a union of horoball quotients in $P_{\eps}^0$ at depth $t$, that is for each $x\in \partial P_{\eps}^t$, $d(x,\partial P_{\eps}^0)=t$. The nonpositive curvature yields that $\vol_s(\beta_t)\leq \vol_s(\beta_0)+ t \vol_{s-1}(\partial \beta_0)$. Now observe by the bound on the norm of the vector field $\norm{\mf{X}}\leq\delta(\Gamma)$, $\phi_t$ moves points at most a distance $\delta t$. Since every point of the support $\partial \beta_t$ is at least distance $t$ to $\partial P_{\eps}^0$, and thus to $M\setminus P_{\eps}$, we have $(\phi_{-t})_*\beta_{\delta(\Gamma)t}$ has a boundary cycle that remains supported in $P_{\eps}$. Since the volume growth of $\phi_{-t}$ is exponentially decreasing in $t$, while that of $\beta_{\delta(\Gamma)t}$ is at most linearly growing in $t$, there is a $T$ such that $\vol_s((\phi_{-T})_*\beta_{\delta(\Gamma)T})<\Theta$, where $\Theta$ is again the constant as in Theorem \ref{thm:small_vanishing} for the scale $\eps/2$. Thus by Theorem \ref{thm:small_vanishing}, $(\phi_{-T})_*\beta_{\delta(\Gamma)T}$ represents zero class in $H_s(M,M_{\eps};\V)$. Finally, since $\beta$, $\beta_{\delta(\Gamma)T}$ and $(\phi_{-T})_*\beta_{\delta(\Gamma)T}$ are homologous to each other in $C_s(M,P_\eps;\V)$, they are also homologous in $C_s(M,M_{\eps};\V)$. It follows that $q_s([\beta])=0$ hence $q_s=0$.
\end{proof}

\begin{remark}
In the first case of the above theorem, the unbounded components of $M_{\eps}$ are often known to be cuspidal. For instance, this occurs if $X$ admits a cocompact lattice of isometries and $M$ is geometrically finite. However, in the general case we are not aware that they need be cuspidal (in the sense above).
\end{remark}

\section{Cohomological Dimension}\label{sec:cohomdim}

In this section, we derive our cohomological results. As noted in the introduction, we will employ the natural flow with Morse theory and dynamics (e.g. Patterson--Sullivan theory).

\subsection{Cohomological Dimension}

By a result of Eilenberg--Ganea (see Thm 7.1 and Cor 7.2 of \cite[Ch VIII]{Brown12}), we have for any group $\Ga$ that $\cd(\Ga)=\cd_{geom}(\Ga)$, except possibly when $\cd(\Ga)=2$, where $\cd_{geom}(\Ga)$ is the smallest dimension of any $K(\Ga,1)$ CW--complex. However, the inequality $\cd(\Ga)\leq\cd_{geom}(\Ga)$ always holds since the standard cellular chain complex gives rise to the projective resolution (as $\mathbb Z\Ga$--module) of $\mathbb Z$. (See Prop 2.2 of \cite[Ch VIII]{Brown12}.)

\subsection{Morse Theory for proper functions}

Recall that given a $C^2$--smooth function $f\colon M\to \R$, the Hessian of $f$ at $p$ is given by $\op{Hess}_p(f)=(\nabla^2f)_p$, that is,
\[ \op{Hess}_p(f)(v,w)=\nabla^2f(v,w)=WV(f)_p-(\nabla_W V)(f)_p=VW(f)_p-(\nabla_VW)(f)_p, \]
where $V,W$ are any smooth vector fields extending vectors $v,w\in T_pM$. In particular, it is a symmetric $2$--tensor that is independent of the choice of extending fields of $v,w$.

We call a function $f$ \textbf{Morse} if all of its critical points are non-degenerate, meaning that its Hessian at each critical point $p$ has full rank as a bilinear form. Non-degenerate critical points are always isolated. Ordinary Morse theory has a natural extension to a noncompact $C^2$--Riemannian manifold under certain proper conditions. A Morse function $f$ is called \textbf{proper} if the preimage of compact sets are compact. It is classical \cite{Milnor} that any smooth function can be ($C^2$) perturbed to be Morse. As is common, for proper functions we may also require that after the perturbation the level sets of $f$ contain at most one critical point. We adopt the following standard notation: for any set $A\subset \R$, set $M^{A}=f^{-1}(A)$ and for $a\in \R$ set $M^a=M^{(-\infty,a]}$. A Morse function detects the topology of the manifold. If $M$ is $n$--dimensional, then for any regular values $a<b$, the sublevel set $M^b$ is obtained from $M^a$ by a finite succession of smooth handle attachments of $D^{k_i}\times D^{n-k_i}$, where $k_1,\dots,k_r$ are the indices of the critical points in $M^{[a,b]}$ ordered by value, and $D^k$ denotes the open unit disk in $\R^k$. We recall the following result from \cite[Thm 3.5]{Milnor}:

\begin{theorem}\label{thm:Morse-theory}
If $f$ is a Morse function on a manifold $M$ (possibly noncompact), and if each $M^a$ is compact, then $M$ has the homotopy type of a CW--complex with one cell of dimension $d$ for each critical point of index $d$.
\end{theorem}

As a consequence, we have,

\begin{theorem}\label{thm:cd_Morse}
If $f$ is a proper Morse function bounded below on a noncompact aspherical manifold $M$ with $\Ga=\pi_1(M)$ whose critical points have index at most $k$, then $\cd(\Ga)\leq k$.
\end{theorem}

\begin{proof} Since $f$ is bounded below, $M^a$ is compact for any $a$. By Theorem \ref{thm:Morse-theory}, $M$ has the homotopy type of a CW--complex $X_f$ of dimension at most $k$. Moreover, since $M$ is aspherical, the homology of $\Ga$ with any given coefficients equals the homology of $M=B\Ga$ with a corresponding module of coefficients. By dimensionality of standard cellular homology with arbitrary $\Z \Gamma$--module coefficients, we have $\cd(\Ga)\leq\dim(X_f)\leq k$. 
\end{proof}

\begin{remark}\label{Rem:LowerBoundMorse}
The uniform lower bound of the Morse function is necessary. Indeed, a proper Morse function may not recover the topology by itself. If we consider the projection $\pi\colon N\times\R\rightarrow \R$ for an $(n-1)$--manifold $N$ then there are no critical points, but the topology (and even cohomology) may not be trivial. If we put a warped product metric on $N\times \R$ with convex $e^{\pm t}$, then the Hessian of the projection can be made everywhere positive semidefinite. Hence the index is $n-1$ but the cohomology does not vanish except on the top dimension. We always have the topology of $M^a$ to contend with, though we can take a colimit as $a\to \infty$. The corresponding cohomology is what we must start with.
\end{remark}

\subsection{Construction of the Morse function} 

We start with the $C^2$--smooth function on $M$ given by $f(x)=-\log\norm{\mu_x}$ as in Section \ref{sec:gradient_flow}, where $\mu_x$ is a $\delta$--conformal density on $X$ associated with $\Ga$. The Hessian was computed as
\begin{align*}
\nabla d{f}(x)&=\frac{\delta\cdot \int_{\pa_\infty X}\left( \nabla dB_\theta-\delta\cdot dB_\theta\otimes dB_\theta\right)(x) d\mu_x(\theta)}{\norm{\mu_x}}+\left(d{f}\otimes d{f}\right)(x).
\end{align*}

Taking the $k$--trace on both sides and use the super-additivity, we obtain for every $x\in M$
\begin{align}\label{eq:Hessian}
\tr_k(\nabla df)(x)&\geq \delta \frac{1}{\norm{\mu_x}} \int_{\pa_\infty X}\left(\tr_k\nabla dB_\theta(x)-\delta \right) d\mu_x(\theta) \geq \delta\cdot \left(\inf_{(x,\theta)\in X\times \partial_\mu X} \op{tr}_k \nabla dB_{(x,\theta)}-\delta\right).
\end{align}
Thus, when $k=j_X(\mu)$, according to the definition, we have $\tr_k(\nabla df)>0$ for all $x\in M$. It follows that the number of negative (or even nonpositive) eigenvalues of $\nabla df$ is bounded by 
$j_X(\mu)$. This shows that we have good index control for our potential function $f$. However, the function $f$ might not be proper and Morse. We will need the following result. (See \S 4 of \cite[Ch 26]{ChowEtAl10}) 
 
\begin{theorem}[{\cite[Thm 1]{Tam09}}]\label{thm:distance_like}
Let $M$ be a complete Riemannian manifold of dimension $n$ whose sectional curvature has bound $\abs{K}\leq \kappa$. Fix ${p}\in M$. Then there is a constant $C=C(n,\kappa)$ and a smooth function $g\colon M\to [0,\infty)$ such that
\begin{enumerate}
\item $d({p},x)+1\leq g(x)\leq d({p},x)+C$,
\item $\norm{\nabla g}\leq C$,
\item $\norm{\nabla d g}\leq C$.
\end{enumerate}
\end{theorem}

\begin{corollary}\label{cor:perturb_Morse}
Suppose $\norm{\mu_x}$ (viewed as a function $M\rightarrow \R^+$) has subexponential growth. Then for any $\eps>0$, there exists a proper Morse function $f_\eps\colon M\to \R$, such that 
\begin{enumerate}
\item $f_\epsilon(x)$ is uniformly bounded from below, and
\item $\norm{\nabla df_\eps-\nabla df}(x)<\eps$ for all $x\in M$.
\end{enumerate}
\end{corollary}

\begin{proof}
Since $\norm{\mu_x}$ has subexponential growth, we know for any $\eta>0$
\[f(x)\geq -\eta \cdot d({p},x)\]
when $d({p},x)$ is large enough. Now we perturb $f(x)$ by setting
\[f'_\epsilon(x):=f(x)+\frac{\epsilon}{2C}\cdot g(x),\]
where $g(x)$ and $C$ are as in Theorem \ref{thm:distance_like}. By setting $\eta=\epsilon/4C$, we have that
\[f'_\epsilon(x)\geq f(x)+\frac{\eps}{2C} \cdot d({p},x)\geq \frac{\eps}{4C}\cdot d({p},x),\]
when $d({p},x)$ is large enough. Thus, $f_\epsilon'$ has a uniform lower bound and is proper. Moreover, we can estimate the Hessian
\[\norm{\nabla df'_\eps-\nabla df}=\norm{\frac{\eps}{2C}\cdot \nabla dg}\leq \frac{\eps}2.\]
Finally, since Morse functions are dense \cite[Ch 6, Thm 1.2]{Hirsch} (in the strong topology) in $C^2(M,\mathbb R)$, we can choose a $C^2$--perturbation $f_\eps$ of $f'_\eps$ so that $f_\eps$ is Morse which remains proper and uniformly bounded from below, whose Hessian satisfies $\norm{\nabla df_\eps-\nabla df}<\eps$.
\end{proof}
 
\begin{reptheorem}{thm:cd}
If $X$ is a Hadamard space and $\Ga<\op{Isom}(X)$ is a discrete, torsion-free subgroup with a conformal density $\mu$ such that $\norm{\mu_x}$ has subexponential growth, then $\cd(\Ga)\leq j_X(\mu)-1$.
\end{reptheorem}

\begin{proof} Let $k=j_X(\mu)$ and $\delta$ be the dimension of the conformal density $\mu$. Set
\[\eps_0=\inf_{(x,\theta)\in X\times \partial_\mu X} \op{tr}_k\nabla dB_{(x,\theta)}-\delta>0.\]
By Equation \eqref{eq:Hessian} we have $\tr_k \nabla d f\geq \delta \epsilon_0$. By Corollary \ref{cor:perturb_Morse}, there exists a proper Morse function $f_\eps$ such that
\[\tr_k\nabla d f_\eps \geq \frac{\delta\eps_0}2>0.\]
Thus the index of $f_\eps$ satisfies $ i(f_\eps)\leq k-1$ at every critical point. By Theorem \ref{thm:cd_Morse}, we obtain $\cd(\Ga)\leq j_X(\mu)-1$.
\end{proof}

\subsection{Subexponential growth of \texorpdfstring{$||\mu_x||$}{}}

We investigate certain situations where $||\mu_x||$ has subexponential growth.
For the rest of this section, we assume $X$ has pinched negative curvature, and for simplicity we normalize such that $K\leq -1$.

\subsubsection{Geometrically finite subgroups}

First we observe that for a convex cocompact subgroup $\Ga<\op{Isom}(X)$, the norm of the unique $\delta(\Ga)$--conformal density $\mu_x$ is always uniformly bounded. To see this, we look at $\nabla f$ where $f=-\log||\mu_x||$ as before. As was computed in Equation \eqref{eq:df}
\[\nabla { f}(x)=\frac{\delta\cdot \int_{\pa_\infty X}\nabla B_\theta(x) d\mu_x(\theta)}{\norm{\mu_x}}.\]
Geometrically, $\nabla B_\theta(x)$ is the unit vector at $x$ pointing in the opposite direction of $\theta$. When $x$ is \emph{not} in the convex hull of the limit set, this direction points towards leaving the convex hull as $\mu_x$ supports on the limit set. Taking the average, we still have that $\nabla f$ is leaving the convex hull. This shows that $f(x)\geq f(P(x))$, where $P(x)$ is the nearest-point projection of $x$ to the convex hull. In other words, we have $||\mu_x||\leq ||\mu_{P(x)}||$. On the other hand, by the cocompactness of the $\Ga$--action on the convex hull, $||\mu_x||$ is uniformly bounded on the convex hull. Thus, $||\mu_x||$ is uniformly bounded in the entire $M$.

The argument almost works for general geometrically finite subgroups, except that the convex hull is now finite volume and $||\mu_x||$ might possibly increase exponentially fast when going into a cusp. Indeed, this must happen (according to Theorem \ref{thm:cd}) when the inequality $\cd(\Ga)\leq j_X(\Ga)-1$ fails. For example, when $\Ga$ is a Schottky-type Kleinian group that is generated by a parabolic element of large rank together with a hyperbolic element of a large translation.

\subsubsection{Bowen--Margulis measures}

Given any oriented geodesic on $M$, it lifts to an oriented geodesic on $X$ up to a $\Ga$--action. Any oriented geodesic on $X$ naturally identifies to an ordered pair of distinct points on $\partial_\infty X$. This gives a one-to-one correspondence
\[\{\text{geodesics on $M$\}}\cong \left((\partial_\infty X\times \partial_\infty X)\setminus \Delta\right)/\Ga,\]
where $\Delta=\{(\theta,\theta)\,:\,\theta\in \partial_\infty X\}$ is the diagonal set. Similarly, the unit tangent bundle $T^1 M$ is naturally identified with $\bigl(((\partial_\infty X\times \partial_\infty X)\setminus \Delta)\times \mathbb R\bigr)/\Ga$. Such an identification is quite useful when investigating $g^t$ (geodesic flow) invariant measures on $T^1M$, since we have
\[\{g^t\text{-invariant measures on }T^1M\}\cong \{\Ga\text{-invariant measures on }(\partial_\infty X\times \partial_\infty X)\setminus \Delta\}.\]

Given any $\delta$--conformal density $\{\mu_x\}$ where $\delta=\delta(\Ga)$ (e.g. using the Patterson--Sullivan construction), we can always associate it with a $\Ga$--invariant measure $U$ on $(\partial_\infty X\times \partial_\infty X)\setminus \Delta$, given by
\[dU(\xi,\eta) = e^{\delta\beta_x(\xi,\eta)} d\mu_x(\xi)d\mu_x(\eta),\]
where $\beta_x(\xi,\eta)=B_\xi(x,z)+B_\eta(x,z)$ for any $z\in [\xi,\eta]$ denotes the Busemann cocycle. Note that the definition of $U$ is independent on $x$, and it is clear that it is $\Ga$--invariant. Thus, it defines a unique $g^t$--invariant measure $\nu$ on $T^1M$, called the \textbf{Bowen--Margulis measure} associated with $\{\mu_x\}$, given by $d\nu(\xi,\eta,t) =dU(\xi,\eta)dt$, under the identification $T^1M\cong \bigl(((\partial_\infty X\times \partial_\infty X)\setminus \Delta)\times \mathbb R\bigr)/\Ga$.

For a general infinite volume discrete subgroup, the Bowen--Margulis measure might be an infinite measure. However, if $\nu(T^1M)<\infty$, then the quotient manifold will have nice dynamical properties. It is shown in \cite[\S 5]{Yue96} that, if the Bowen--Margulis measure is finite, then the geodesic flow $g^t$ is ergodic and conservative. It follows that \cite[\S 3]{Yue96} $\Ga$ is of divergence type and $\{\mu_x\}$ is the unique $\delta(\Ga)$--conformal density on $M$ which is non-atomic. In the same paper, he also gave a sufficient condition \cite[Thm 5.4.3]{Yue96} for the finiteness of Bowen--Margulis measure.

More generally, Pit--Schapira \cite{PitSchapira18} gave three sufficient conditions as well as necessary conditions for any Gibbs measure to be finite and also recover the necessary and sufficient condition given by Dal'bo--Otal--Peign\'e in \cite{DalboEtAl00} in the geometrically finite case. 

We now relate the finiteness of Bowen--Margulis measure to the cohomological dimension.

\begin{reptheorem}{thm:finiteness criterion}
 Suppose $X$ has pinched negative curvature and $M=X/\Gamma$ is a manifold whose injectivity radius is bounded away from zero. If $\nu(T^1M)<\infty$, then the corresponding unique $\delta(\Ga)$--conformal density $\mu_x$ has uniformly bounded norm on $M$.
\end{reptheorem}

Combining Theorem \ref{thm:finiteness criterion} with Theorem \ref{thm:cd}, we obtain the following corollary.

\begin{repcorollary}{cor:CDBound}
Suppose $X$ has pinched negative curvature and $M=X/\Gamma$ is a manifold whose injectivity radius is bounded away from zero. If $\nu(T^1M)<\infty$, then $\op{cd}(\Ga)\leq j_X(\mu)-1$, 
where $\mu$ is the unique $\delta(\Ga)$--conformal density associated to $\nu$. Furthermore, if $X$ is normalized to be $K\leq -1$, then $\op{cd}(\Ga)\leq \floor{\delta(\Ga)}+1$.
\end{repcorollary}

\begin{proof}
The first part follows from Theorem \ref{thm:finiteness criterion} and Theorem \ref{thm:cd}, and the second part follows from the fact that $j_X(\mu)\leq \floor{\delta(\Ga)}+2$ as in the proof of \cref{T:MainJac}.
\end{proof}

The rest of the section is devoted to the proof of Theorem \ref{thm:finiteness criterion}.

\subsubsection{Uniformly exponentially bounded measures}

We now consider the lift of the Bowen--Margulis measure $\nu$ to the universal cover $T^1X$. The resulting measure is $\Ga$--invariant and that we denote by $\nu$ for simplicity.

\begin{definition}\label{def:UEB}
We say the Bowen--Margulis measure $\nu$ on $T^1M$ is \textbf{uniformly exponentially bounded} if there exists a constant $C>0$ such that $\nu$ satisfies
\[\nu(T^1(B(x,R)))\leq C e^{C R}\]
for any $x\in X$ and $R>0$, and $B(x,R)$ is the radius $R$ metric ball on the universal cover $X$ of $M$.
\end{definition}

The following lemma gives a class of manifolds on which the Bowen--Margulis measures are uniformly exponentially bounded.

\begin{lemma}\label{lem:counting}
If the injectivity radius of $M$ is uniformly bounded from zero and $\nu(T^1M)<\infty$, then $\nu$ is uniformly exponentially bounded.
\end{lemma}

\begin{proof}
Denote $\pi\colon X\rightarrow M$ the projection map. Given any $x\in X$ and $R>0$, we look at the number $N(x,R)$ of Dirichlet fundamental domains in $X$ that intersect with $B(x,R)$. For any Dirichlet fundamental domain parameterized by an element $\ga\in\Ga$, that is, 
\[\mathfrak F_\ga:=\{y\in X\,:\,d(y,\ga x)\leq d(y,g x)\text{ for any }g\in \Ga\},\]
if it intersects with $B(x,R)$, then by triangle inequality $d(x,\ga x)\leq 2R$. Denote $\epsilon_0>0$ a lower bound of the injectivity radius, then we know that $\bigcup_\ga B(\ga x,\epsilon_0)$ are disjoint unions of metric balls in $X$. Furthermore, if $\mathfrak F_\ga$ intersects with $B(x,R)$, then $\ga x\in B(x,2R)$ hence $B(\ga x,\epsilon_0)\subset B(x,2R+\epsilon_0)$. Since $X$ has pinched negative curvature, there exists a uniform constant $C_0$ (independent of $x$) such that $\vol(B(x,\epsilon_0))\geq C_0$. Therefore, we have
\begin{align*}
N(x,R)\cdot C_0&\leq \sum_{\{\ga\,:\, \mathfrak F_\ga \text{ intersects with }B(x,R)\}}\vol(B(\ga x,\epsilon_0))\\
&\leq \vol(B(x,2R+\epsilon_0)) \leq C_1 e^{C_2(2R+\epsilon_0)},
\end{align*}
for some $C_1, C_2$ that only depends on $X$. Thus,
\[\nu(T^1 (B(x,R)))\leq N(x,R)\cdot \nu(T^1M)\leq \frac{C_1}{C_0}\cdot e^{C_2(2R+\epsilon_0)}\cdot \nu(T^1M).\]
 It follows that $\nu$ is uniformly exponentially bounded.
\end{proof}

\subsubsection{Estimate of gradient norm}

We intend to understand how $||\mu_x||$ varies on $M$ by understanding how it increases along its gradient field. We prove that, once $||\mu_x||$ exceeds a threshold value, then it will continue to increase to infinity exponentially fast and with the maximal exponential growth rate $\delta$. This is done via an ordinary differential inequality.

\begin{lemma}(Small shadow)\label{lem:shadow}
If $\nu$ is uniformly exponentially bounded, then for any $0<\beta<1$, there exists $\eta=(1-\beta)/C>0$ such that for any $x\in X$ with $\norm{\mu_x}>e^{1/\eta}$, we have
\[\inf_{\xi\in \partial_\infty X} \mu_{x}(Sh_{\xi}(B( x,R)))\leq \frac{C}{\norm{\mu_x}^\beta}\]
where $R=\eta\cdot \log\norm{\mu_x}$ and $C$ is the constant as in Definition \ref{def:UEB} and only depends on $M$.
\end{lemma}

\begin{proof} 
Consider $B(x,R)\subset B(x,2R)$. Since $\nu$ is uniformly exponentially bounded, we have
\[\nu(T^1(B(x,2R)))\leq C e^{CR}.\]
On the other hand, if we denote $A(R)\subset T^1(B(x,2R))$ the set of all unit vectors whose corresponding geodesics intersect with $B(x,R)$, that is,
\[A(R)=\{v\in T^1(B(x,2R)): g^t(v)\cap T^1(B(x,R))\neq \emptyset \text{ for some } t\in \R\},\]
then we can estimate
\begin{align*}
\nu(T^1(B(x,2R)))&\geq \nu(A(R)) \geq\int_{\pa_\infty X} \int_{Sh_\xi(B(x,R))}\left(\int_{T_1}^{T_2} e^{\delta \beta_{x}(\xi,\eta)}dt\right)d\mu_{x}(\eta)d\mu_{x}(\xi)\\
&\geq \int_{\pa_\infty X} \int_{Sh_\xi(B(x,R))} (T_2-T_1)d\mu_{x}(\eta)d\mu_{x}(\xi)\\
& \geq \int_{\pa_\infty X}\mu_{x}(Sh_{\xi}(B(x,R))) d\mu_{x}(\xi) \geq \inf_{\xi\in \partial_\infty X} \mu_{x}(Sh_{\xi}(B(x,R)))\cdot \norm{\mu_{x}},
\end{align*}
where in the second inequality, $T_1, T_2$ are the time variable of the geodesic $\xi\eta$ entering and leaving $B(x, 2R)$. As the geodesic intersects $B(x,R)$, we have $T_2-T_1\geq 2\sqrt 3 R\geq 1$ by assumption. This yields the fourth inequality. We also used $\beta_x(\xi,\eta)\leq 0$ in the third equality. Combining the above two inequalities gives us
\[\inf_{\xi\in \partial_\infty X} \mu_{\til x}(Sh_{\xi}(B(\til x,R)))\leq \frac{C}{\norm{\mu_x}^{1-C\eta}}.\]
Thus, by choosing $\eta=(1-\beta)/C$, the lemma follows.
\end{proof}

\begin{lemma}\label{lem:gradient-norm}
Suppose $\nu$ is uniformly exponentially bounded. Denote the function $f(x)=\norm{\mu_x}$. Then there exist constants $\alpha\in (0,1)$, $D(M)>0$ that depend only on $M$ such that if $f(x)\geq D(M)$, then
\[\norm{\nabla f}(x)\geq \delta f(x)-f^\alpha(x).\]
In particular, if the initial condition $f(x_0)$ is large enough for some $x_0\in M$, then under the gradient flow (integration of the $C^1$--smooth field $\nabla f$) (for $t>0$), the trajectory does not hit any critical point and will escape to infinity.
\end{lemma}

\begin{proof}
We do the computations by lifting to the universal cover but by the $\Gamma$--equivariance, all the geometric quantities descends to the quotient manifold. We first compute the gradient,
\begin{equation}\label{eq:gradient}
\nabla f(x)=\delta \int_{\pa_\infty X}v_{x\theta} d\mu_x.
\end{equation}
By lemma \ref{lem:shadow}, for any $x\in M$, and any lift $\til x\in X$, there is $\xi\in \pa_\infty X$ that attains the infimum of the Patterson--Sullivan mass of the shadow. If $v$ is the unit vector from $\til x$ pointing towards $\xi$, then
\[\norm{\nabla f}(\til x)=\sup_{\theta\in \pa_\infty X}\langle\nabla f, v_{\til x\theta}\rangle_{\til x}\geq \langle\nabla f(\til x), v\rangle= \delta\int_{\pa_\infty X}\langle v_{\til x\theta}, v\rangle d\mu_{\til x}(\theta).\]
We write $\pa_\infty X$ as the disjoint union of $Sh_{\xi}(B(\til x,R))$ and its complement where $R=\eta\cdot \log f$ as in Lemma \ref{lem:shadow}. For each point $\theta$ in the shadow complement, the angle between $v$ and $v_{\til x\theta}$ is very small when $R$ is large. More precisely, by the hyperbolic law of cosine, we obtain that the angle $\alpha$ between $v$ and $v_{\til x\theta}$ has to satisfy
\begin{equation}\label{eq:angle}
\sin(\alpha/2)\leq \frac{1}{\cosh R},
\end{equation}
and so
\[\langle v_{\til x\theta}, v\rangle=1-2\sin^2(\alpha/2)\geq 1-\frac{2}{\cosh^2 R}\]
for $\theta\in Sh_{\xi}(B(\til x,R))^c$.
Now we can estimate the norm of the gradient,
\begin{align*}
\norm{\nabla f}(\til x)&\geq \delta\left(\int_{Sh_{\xi}(B(\til x,R))}\langle v_{\til x\theta}, v\rangle d\mu_{\til x}(\theta)+\int_{Sh_{\xi}(B(\til x,R))^c}\langle v_{\til x\theta}, v\rangle d\mu_{\til x}(\theta)\right)\\
&\geq \delta\left((-1)\cdot\mu_{\til x}(Sh_{\xi}(B(\til x,R)))+(1-\frac{2}{\cosh^2 R})\cdot \mu_{\til x}(Sh_{\xi}(B(\til x,R))^c)\right)\\
&\geq \delta\left(-\frac{C(M)}{\norm{\mu_x}^\beta}+(1-\frac{2}{\cosh^2 R})(\norm{\mu_x}-\frac{C(M)}{\norm{\mu_x}^\beta})\right)  \textrm{  (By Lemma \ref{lem:shadow})}\\
&\geq \delta \left(f-8f^{1-2\eta}-2C(M)f^{-\beta}+8C(M)f^{-(2\eta+\beta)}\right),
\end{align*}
where the last inequality uses $\cosh R\geq e^R/2$ and $R=\eta \log f$. Now we specify $\beta=1/2$ and $\eta=\frac{1}{2C}$ as in Lemma \ref{lem:shadow}. We can choose $\alpha=1-\eta$ and $D(M)$ large enough so that
\[f^\alpha\geq \delta(8f^{1-2\eta}+2C(M)f^{-\beta}),\]
whenever $f\geq D(M)$. Thus it follows that $\norm{\nabla f}\geq \delta f-f^\alpha$.
\end{proof}

\begin{lemma}\label{lem:ODE}
Let $y(t)$ be a positive function that satisfies the differential inequality
\[Cy\geq y'\geq C y-y^\alpha\]
for some constants $C>0$ and $0<\alpha<1$. Then
\[y(t)\geq \left(y(0)-\frac{y(0)^\alpha}{C(1-\alpha)}\right)e^{Ct}.\]
\end{lemma}

\begin{proof} 
Let $x(t)=y(t)/e^{Ct}$, we compute
\[x'(t)=\frac{y'(t)e^{Ct}-Cy(t)e^{Ct}}{e^{2Ct}}.\]
It follows that
\[-\frac{y^\alpha}{e^{Ct}}\leq x'(t)\leq 0.\]
The second inequality implies $y(t)\leq y(0)e^{Ct}$, together with the first inequality, it implies
\[x'(t)\geq -y(0)^\alpha e^{-(1-\alpha)Ct}.\]
Therefore, for any $t>0$,
\begin{align*}
x(t)-x(0)=\int_{0}^t x'(s)ds&\geq \int_{0}^t -y(0)^\alpha e^{-(1-\alpha)Cs}ds\\
&=-\frac{y(0)^\alpha}{C(1-\alpha)}\cdot (1-e^{C(1-\alpha)t}) \geq -\frac{y(0)^\alpha}{C(1-\alpha)}.
\end{align*}
Thus,
\[y(t)\geq \left(y(0)-\frac{y(0)^\alpha}{C(1-\alpha)}\right)e^{Ct}.\]
\end{proof}

\subsubsection{Proof of Theorem \ref{thm:finiteness criterion}} 

We are now ready to prove Theorem \ref{thm:finiteness criterion}. We prove a slightly more general result (Theorem \ref{thm:ubd} below). Before stating the theorem, we make an observation.

\begin{proposition}
If $X$ has pinched negative curvature and $||\mu_x||$ is uniformly bounded, then $\nu$ is uniformly exponentially bounded.
\end{proposition}

\begin{proof}
We have
\begin{align*}
\nu(T^1(B(x,R)))&=\int_{\pa_\infty X} \int_{\pa_\infty X}\left(\int_{T_1}^{T_2} e^{\delta \beta_{x}(\xi,\eta)}dt\right)d\mu_{x}(\eta)d\mu_{x}(\xi)\\
&\leq \int_{\pa_\infty X} \int_{\pa_\infty X} (T_2-T_1)\cdot e^{2\delta R}d\mu_{x}(\eta)d\mu_{x}(\xi) \leq 2R\cdot e^{2\delta R}||\mu_x||^2 \leq Ce^{CR},
\end{align*}
where $T_1, T_2$ represent the time parameters of the geodesic $\xi\eta$ that enters and leaves $B(x,R)$ respectively, and by triangle inequality $T_2-T_1\leq 2R$. The second line of the inequality also uses the triangle inequality for $\beta_x(\xi,\eta)\leq 2R$ as $\xi\eta$ is at most $R$ distance to $x$. Finally, the last inequality uses that $||\mu_x||$ is uniformly bounded and $C$ is chosen large enough and does not depend on $x$.
\end{proof}

We now show that the converse is also true, with a mild extra condition that $\mu_x$ is non-atomic. The idea is to argue under the condition that $||\mu_x||$ grows exponentially fast at the maximal rate $\delta$, hence the trajectory limiting to $\partial_\infty X$ will form a Dirac mass.

\begin{theorem}\label{thm:ubd}
If $X$ has pinched negative curvature, $\nu$ is uniformly exponentially bounded, and $\mu_x$ is non-atomic, then $\norm{\mu_x}$ is uniformly bounded on $M$.
\end{theorem}

\begin{proof}
Suppose that $\norm{\mu_x}\rightarrow \infty$. Then there exists $x_0\in M$ such that $\norm{\mu_{x_0}}\geq D(M)$ and
\begin{equation}\label{eq:large-initial}
\norm{\mu_{x_0}}-\frac{\norm{\mu_{x_0}}^\alpha}{\delta(1-\alpha)}\geq 1,
\end{equation}
where $D(M)$ and $\alpha$ are constants as in Lemma \ref{lem:gradient-norm}. Consider the gradient flow $\Phi(t)$ of $f(x)=\norm{\mu_x}$ with re-parameterized length parameter (unit speed) $t$, initiated at $x_0$. Under such normalization, we have $f'(\Phi(t))=\langle \nabla f, \Phi'(t)\rangle=\norm{\nabla f}$ since $\Phi'(t)$ is in the same direction as $\nabla f$ but with unit speed. Under the gradient flow, $\norm{\mu_x}$ is non-decreasing, and because of the escaping lemma \cite[Lemma 9.19]{Lee}, the flow is well defined for all $t\geq 0$. According to Lemma \ref{lem:gradient-norm}, we have
\[\frac{d}{dt}f(\Phi(t))=\norm{\nabla f}(\Phi(t))\geq \delta f(\Phi(t))-f(\Phi(t))^\alpha,\]
for any $t\geq 0$. On the other hand, it is clear from equation \eqref{eq:gradient} that $\norm{\nabla f}(\Phi(t))\leq  \delta f(\Phi(t))$. Thus $y(t)=f(\Phi(t))$ satisfies the differential inequality of Lemma \ref{lem:ODE}, and it follows that
\[y(t)=\norm{\mu_{\Phi(t)}}\geq \left(\norm{\mu_{x_0}}-\frac{\norm{\mu_{x_0}}^\alpha}{\delta(1-\alpha)}\right)e^{\delta t}\geq e^{\delta t}\]
in view of inequality \eqref{eq:large-initial}. Since $t$ is the length parameter, we obtain that
\begin{equation}\label{eq:exponential}
\norm{\mu_{\Phi(t)}}\geq e^{\delta d(\Phi(t),x_0)},\quad \text{ for all } t\geq 0.
\end{equation}
Next, we claim $\mu_{x_0}$ must have an atomic mass. To see this, we investigate further the shadow complement $Sh_{\xi}(B(x,R))^c$ for each $x$ on $\Phi(t)$. 
	
We denote $v_t$ the unit vector at $\Phi(t)$ pointing in the opposite direction to the geodesic ray connecting $\Phi(t)$ and $\Phi(0)=x_0$, and $\xi_t\in \pa_\infty X$ be any point achieving the following infimum
\[\inf_{\xi\in \pa_\infty X}\mu_{\Phi(t)}(Sh_{\xi}(B(\Phi(t),R_t)))\]
at $x=\Phi(t)$ as in Lemma \ref{lem:shadow}. Let $A(t)\subset \pa_\infty X$ denote the shadow complement and $u_t$ be the unit vector at $\Phi(t)$ pointing towards $\xi_t$. Since $\Phi(t)$ leaves any compact set, there exists a sequence of points $t_i\rightarrow \infty$ such that the angle between $v_{t_i}$ and $\Phi'(t_i)$ is less than $\pi/2$. On the other hand, the angle between $\Phi'(t)$ and $u_{t}$ tends to zero. This is because Lemma \ref{lem:shadow} implies $\mu_{\Phi(t)}$ has large concentration on $A(t)$ and as $R_t\rightarrow \infty$ equation \eqref{eq:angle} implies the diameter of $A(t)$ tends to zero under the spherical metric at $\Phi(t)$. So $\nabla f(\Phi(t))$ (same direction as $\Phi'(t)$) which represents the average of $v_{x\theta}$ is dominated by the direction towards $A(t)$ (or $u_{t}$). Thus, by possibly passing to a subsequence, we may assume $\Phi'(t_i)$ is close enough to $u_{t_i}$ hence $v_{t_i}$ and the entire $A(t_i)$ has angle $<3\pi/4$ viewed from $\Phi_{t_i}$. Moreover we may assume
\begin{equation}\label{eq:half-concentrate}
\mu_{\Phi(t_i)}(A(t_i))\geq \norm{\mu_{\Phi(t_i)}}/2.
\end{equation}
Now the pinched negative curvature implies that when viewed from $x_0$, the diameter of the set $A(t_i)$ also tends to zero. On the other hand, we can estimate
\begin{align*}
\mu_{x_0}(A(t_i))&=\int_{A(t_i)}e^{-\delta B_{\xi}(x_0,\Phi(t_i))}d\mu_{\Phi(t)}(\xi) \geq e^{-\delta d(x_0,\Phi(t_i))}\cdot \mu_{\Phi(t_i)}(A(t_i)) \geq 1/2,
\end{align*}
where the last inequality uses inequalities \eqref{eq:exponential} and \eqref{eq:half-concentrate}. This means $\bigcap_{i}A(t_i)$ has to be a point where $\mu_{x_0}$ has an atomic mass at least $1/2$. This gives a contradiction.
\end{proof}

\begin{proof}[Proof of Theorem \ref{thm:finiteness criterion}]
Note that $\nu(T^1M)<\infty$ implies that $\Ga$ is of divergence type and $\mu$ is non-atomic \cite[Prop. 3.5.4]{Yue96}. The rest follows from Lemma \ref{lem:counting} and Theorem \ref{thm:ubd}.
\end{proof}

\appendix

\section{Proof of Theorem \ref{thm:small_vanishing}}\label{sec:proof_of_small}\label{app:GromovTheorem}

To begin with, we observe that we may construct a simplicial model homotopy equivalent to $M=X/\Ga$ by taking the nerve of a cover by contractible open sets with all intersections of these sets contractible. However, with more care we can choose such a simplicial model that is uniformly Lipschitz homotopic to $M$. To realize this, first scale the metric on $M$ so that the sectional curvatures satisfy $-1\leq K\leq 0$. Consider the simplicial nerve $\mc{N}$ of a cover $\mathcal{U}$ guaranteed by the following proposition. The necessary modifications to the proof of this proposition for the case of general non-positively curved manifolds is routine from the existence of normal neighborhoods and the small radius doubling properties of the Riemannian volume measure when curvature is bounded below. We recall that we defined $M_{\eps}=\set{x\in M\,:\, \op{injrad}(x)<\eps}$. Set $M_{\geq \eps}=M\setminus M_{\eps}$ to be the complement of $M_{\eps}$. We alert the reader that in \cite{Kapovich} the author used $M_{\eps}$ to denote our $M_{\frac{\eps}{2}}$.

\begin{proposition}[{\cite[6.4]{Kapovich}}]\label{prop:good-cover} 
There exists a function 
\[ m(n, \epsilon)\colon  \mathbb{N} \times(0, \infty) \longrightarrow \mathbb{N}, \] with the following property. For every complete $n$--manifold $M$ with $-1\leq K_M\leq 0$, there exits a countable subset $E=\left\{x_{i}, i \in I\right\} \subset M$ and a collection of positive numbers $\left\{\rho_{i}, i \in I\right\}$, so that:
\begin{enumerate}
\item Set $\mathcal{D}:=\left\{D_{i}=B_{\rho_{i} / 2}\left(x_{i}\right): x_{i} \in E\right\}$ and $\mathcal{U}:=\left\{B_{i}=B_{\rho_{i}}\left(x_{i}\right): x_{i} \in E\right\}$.
Then $\mathcal{D}$ (and therefore $\mathcal{U}$) covers $M$.
\item For every $x_{i} \in M_{\geq\eps}$, we have
$\rho_{i}=\frac{\epsilon}{16}$.
\item For every $x_{i} \in M_{\eps}$, we have $\rho_{i}=\frac{\op{injrad}\left(x_{i}\right)}{16} \leq \frac{\epsilon}{16}$.
\item The multiplicity of the covering $\mathcal{U}$ is at most $m(n, \epsilon)$.
\end{enumerate}
\end{proposition}

The proof of the following lemma holds verbatim in our context.
\begin{lemma}[{\cite[Lemma 6.6]{Kapovich}}]\label{lem:partition} There exists a smooth partition of unity $\left\{\eta_{i}, i \in \mathbb{N}\right\}$ subordinate to the covering $\mathcal{U}$, so that every function $\eta_{i}$ is $l_{i}$--Lipschitz, with
\[ L_{\eps}:=\sup \left\{l_{i}: x_{i} \in M_{\geq\eps}\right\}<\infty \]
for every $\eps>0$.
\end{lemma}

Observe that identifying $B_i$ with the standard basis vector $e_i\in\bigoplus_{i\in \N}\R$, induces a natural identification of $\mc{N}$ as a subset of the pre-Hilbert simplex,
\[	\begin{gathered} \Delta^{\infty}:=\left\{z \in\bigoplus_{i \in \mathbb{N}} \mathbb{R}: \sum_{i=1}^{\infty} z_{i}=1, z_{i} \geq 0, i \in \mathbb{N}\right\}. \end{gathered} \]
Let $\set{\eta_i}$ be a smooth partition of unity subordinate to the cover $\mc{U}$ given by Lemma \ref{lem:partition}. Consider the map $\eta\colon M\to \mc{N}\subset \Delta^\infty$ defined by 
\[ \eta(x)=(\eta_1(x),\eta_2(x),\dots). \]
As the partition of unity is uniformly $L_\eps$--Lipschitz on $M_{\geq\eps}$ we obtain that $\eta$ is piecewise smooth on simplices and for every $\eps>0$, the restriction $\eta\rest{M_{\geq\eps}}$ is $\sqrt{m(n,\eps)}L_\eps$--Lipschitz.

Note by property (3) of Proposition \ref{prop:good-cover}, and the triangle inequality, the union of the balls in $\mc{U}$ that intersects a given ball in $B_{\rho_i}(x_i)\in\mc{U}$ is contained in a (convex) normal neighborhood of $x_i$. In particular, each $B_{\rho_i}(x_i)$ is convex and intersections of any such balls are therefore convex and contractible so $\eta$ is a homotopy equivalence by the standard Nerve Theorem of Leray. However, we need a controlled homotopy inverse to $\eta$. As pointed out in \cite[Rem 6.8]{Kapovich}, this is the one point in his argument where Kapovich uses the assumption that he is working on a constant curvature manifold. Hence, we will need to modify his construction by using the barycenter map instead. 

For a complete $\CAT(0)$--space $B$, and a compactly supported measure $\mu$ with $\supp(\mu)\subset B$, we define $\bary(\mu)\in B$ of $\mu$ to be the unique $y \in B$ minimizing the (strictly) convex function 
\[ \mc{B}_\mu(y)=\int_B d(x,y)^2 d\mu(x). \]  
The existence of $\bary(\mu)\in B$ follows from the fact that distance functions are unbounded on unbounded sets but $\mathrm{supp}(\mu)$ is bounded. (See \cite[Prop 4.3]{Sturm03} for a more general statement on existence and uniqueness.) We claim that $\bary(\mu)$ lies in $\textrm{Hull}(\supp(\mu))\subset B$, the convex hull of the support of $\mu$. To see this, note that for any point $z\in B\setminus \mathrm{Hull}(\supp(\mu))$ and each point $x\in\supp(\mu)$, we have $d(z,x)>d(\bar{z},x)$ where $\bar{z}$ is the unique closest point in $\mathrm{Hull}(\supp(\mu))$ to $z$. This follows by the law of cosines for the $\CAT(0)$--metrics as the triangle $\Delta(y,\bar{y},x)$ has an obtuse Alexandrov angle at $\bar{z}$. 

Let $\Delta_{J}=\left[e_{j_{0}}, \ldots, e_{j_{k}}\right]$ be a $k$--simplex in $\mc{N}, J=\left\{j_{0}, \ldots, j_{k}\right\} \subset \N$. Recall that for $j\in J$ the vertices $e_{j}$ correspond to the balls $B_{j}=B_{\rho_{j}}\left(x_{j}\right) \in \mathcal{U}$. We define the map $\bar{\eta}\colon \Delta_{J} \longrightarrow M$ first on vertices by sending the vertex $e_{j}$ to the corresponding center $x_{j} \in B_{j}$. Again by property (3) of Proposition \ref{prop:good-cover}, the union of the balls $\bigcup_{j \in J} B_{j}$ is contained in a closed convex neighborhood $B\subset M$. Indeed, we could take $B$ to be the ball of radius $\frac{\op{injrad(x_j)}}{4}$ about any of the $x_j$ for $j\in J$. Though $M$ is only locally $\CAT(0)$, $B$ is a complete $\CAT(0)$ space and we will apply the above barycenter construction to atomic measures supported there. 

For each point $x\in \Delta_{J}\subset \mc{N}$ let $(a_0(x),\dots,a_{k}(x))$ be its Euclidean barycentric coordinates. We define the map $\bar\eta:\Delta_J\rightarrow M$ as
\[ \bar{\eta}(x):=\bary\left( \sum_{i=0}^k a_i(x)\delta_{x_i} \right) \] 
where $\delta_{x_i}$ is the Dirac measure at the point $x_i\in M$ corresponding to the identified vertex of the nerve $\mc{N}$. 
In this case the map $\bar \eta(x)$ is the unique minimum of the function 
\[ y\mapsto \sum_{i=0}^k a_i(x)d(x_i,y)^2. \]

We observe that $\bar{\eta}$ is globally well defined and continuous on all of $\mc{N}$ since a sub-face corresponds to a subset of the $\set{a_i}$ vanishing and the barycentric coordinate parameterizations are compatible with combinatorics of the face matchings. Denote $V_J$ the convex hull of the vertex set $\set{x_j}_{j\in J}$. It follows from the above discussion of the barycenter that $\bar\eta(x)$ lies in $V_J$. We claim 
\[ V_J\subset \bigcap_{j\in J}B_r(x_j),\]
where $r=2\max\set{\rho_{j_0},\dots,\rho_{j_k}}$. Without loss of generality, we assume $\rho_{j_0}\leq \cdots \leq \rho_{j_k}$. First, we notice that for any $i,j\in J$,
\begin{equation}\label{eq:distance}
d(x_i,x_j)\leq \rho_i+\rho_j\leq 2\rho_{j_k}=r. 
\end{equation}
It follows by the triangle inequality that $B_{r}(x_i)\subset B_{2r}(x_k)$. Note that $2r=4\rho_k\leq \frac{1}4\op{injrad}(x_k)$, so $B_{2r}(x_k)$ is convex. For any other point $x\in M$ we have as a consequence of the triangle inequality that $\op{injrad}(x_k)\leq d(x_k,x)+\op{injrad}(x)$. Hence, we have $\op{injrad}(x_i)\geq 8r-r=7r$ for each $i\in J$, even when $x_k\in M_{\geq \eps}$ and $x_i\in M_{\eps}$. Therefore each $B_{r}(x_i)$ ($i\in J$) is convex and so is the intersection $\bigcap_{j\in J}B_r(x_j)$. On the other hand, we see from \eqref{eq:distance} that each $B_r(x_j)$ contains the vertex set $\set{x_j}_{j\in J}$. Thus by convexity, each $B_r(x_j)$ contains $V_J$ and so does the intersection. This proves the claim.

Since we have shown that the barycenter map $\bar{\eta}$ maps $\Delta_J$ into the convex hull of the corresponding vertices which further lies inside $\bigcap_{j\in J}B_r(x_j)$ for the above given $r$, we can prove a similar lemma to \cite[Lemma 6.9]{Kapovich}.

\begin{lemma}
If $x \in M$ and $z \in \operatorname{Star}(\eta(x))$, then $d(x, \bar{\eta}(z)) \leq \frac{\eps}{4}$. Moreover, there exists a homotopy $H$ between $\bar{\eta} \circ \eta$ and $I d$, whose tracks have length at most $\frac{\eps}{4}$.
\end{lemma}

\begin{proof}
The proof essentially follows from \cite[Lemma 6.9]{Kapovich}. Denote $\Delta_J\subset \mathcal N$ the smallest simplex containing $\eta(x)$. Since $z \in \operatorname{Star}(\eta(x))$, there exists a maximal index set $J'\supset J$ spanning a maximal simplex in the nerve complex such that both $z, \eta(x)$ are in $\Delta_{J'}$. Let $r'=\max\set{\rho_j}_{j\in J'}$ and let $x_i$ ($i\in J$) be a vertex so that $x\in B_i$. Then we have
\[d(x,\bar\eta(z))\leq d(x,x_i)+d(x_i,\bar\eta(z))\leq \rho_i+r'\]
since $\bar{\eta}(z)\in V_{J'}\subset \bigcap_{j\in J'}B_r(x_j)$. As both $\rho_i$ and $r'$ are at most $\frac{\eps}8$, we have $d(x, \bar{\eta}(z)) \leq \frac{\eps}{4}$, thus proving the first assertion. For the second assertion, we take the geodesic homotopy between $\bar\eta \circ \eta$ and $Id$. For any $x\in M$, denote $\Delta_J\subset \mathcal N$ the smallest simplex containing $\eta(x)$. Then both $\bar\eta (\eta(x))$ and $x$ are in the convex set $\bigcap_{j\in J}B_r(x_j)$. Thus, the geodesic homotopy is well defined. The fact that the tracks of the homotopy has length at most $\frac{\eps}4$ follows from the first assertion.
\end{proof}

The final proposition we will need has a proof which follows verbatim that of \cite[Prop 6.10]{Kapovich}, using the above (modified) results in the corresponding places of the proof. (Observe that in the Kapovich reference the result is stated in terms of a constant $\kappa$, but that can be taken to be $\kappa=\frac{\eps}{4}$ giving the formulation below.) Set $\mc{N}_\eps$ to be the star of $\eta(M_{\eps})\subset\mc{N}$ and as $\eta:M\to\mc{N}$ induces an isomorphism $\pi_1(M)\to\pi_1(\mc{N})$ we have a flat bundle $\bb{W}$ on $\mc{N}$ corresponding to the $R\Gamma$--module $V$.

\begin{proposition}\label{prop:epsto2eps}

\[ \mathrm{ker}(\eta_{\epsilon, *}\colon H_*\left(M, M_{\epsilon}; \mathbb{V}\right) \rightarrow H_*\left(\mathcal N, \mathcal N_\epsilon ; \mathbb{W}\right)) <
\mathrm{ker}(H_*\left(M, M_{\epsilon}; \mathbb{V}\right) \rightarrow H_*\left(M, M_{\frac32 \epsilon}; \mathbb{V}\right)). \]
\end{proposition}

The rest of the proof of \cref{thm:small_vanishing} now follows the identical argument given in the proofs of \cite[Thm 7.1]{Kapovich} and its intermediary steps: Proposition 7.2, Lemma 7.3, and Lemma 7.4 from \S 7 of \cite{Kapovich} with one important caveat that we now explain. 

In the second to last line of the proof of Theorem 7.1 of \cite{Kapovich}, the inclusion of pairs does not induce an isomorphism $H_q(M,M_{\eps};\V)\to H_q(M,M_{\epsilon_0};\V)$ in our case. Indeed, this is not an isomorphism even in the constant curvature $-1$ case because of the possible presence of small Margulis tubes as Example \ref{ex:lochness} shows. However, in the case of constant curvature it will be an isomorphism for $q>2$. This follows from the long exact sequence for pairs, the fact that Margulis tubes have cohomological dimension at most 1, and that the inclusion of cusps components in $M_{\eps}$ into the corresponding cusps components in $M_{\epsilon_0}$ are strong deformation retracts. For this reason we do not achieve the sharper vanishing into $H_q(M,M_{\eps};\V)$. However, if the inclusion happens to induce an isomorphism $H_q(M,M_{\eps};\V)\to H_q(M,M_{\frac32\eps};\V)$, we do obtain this sharper vanishing. 

We finally note that the vanishing of a class in $H_q(M,M_{\frac32\eps};\V)$ implies vanishing in $H_q(M,M_{2\eps};\V)$ under the map induced by inclusion. For convenience we use the scale $2\eps$ in the statement and proof of \cref{thm:small_vanishing}.


\bibliographystyle{alpha}

\footnotesize{
CC: Indiana University, Bloomington, IN, USA. E-mail \verb|connell@indiana.edu|

DBM: Purdue University, West Lafayette, IN, USA. E-mail: \verb|dmcreyno@purdue.edu|

SW: ShanghaiTech University, Pudong, Shanghai, China. E-mail: \verb|wangshi@shanghaitech.edu.cn|}

\end{document}